\documentclass[11pt]{amsart}

\usepackage{graphicx}
\usepackage{xcolor}
\usepackage{amsfonts}
\usepackage{amsmath}
\usepackage{amssymb}
\usepackage{mathrsfs}
\usepackage{enumerate}
\usepackage[margin=2cm]{geometry}
\usepackage{url}
\usepackage[hidelinks]{hyperref}
\usepackage{comment}
\usepackage{cite}

\DeclareMathOperator*{\ddim}{dim}

\DeclareMathOperator*{\psh}{PSH}
\DeclareMathOperator*{\ddc}{dd^c}
\DeclareMathOperator*{\Vdm}{Vdm}

\DeclareMathOperator*{\sign}{\mathrm{sgn}}
\DeclareMathOperator*{\trace}{trace}
\DeclareMathOperator*{\Span}{span}

\newcommand{\R}{\mathbb R}
\newcommand{\N}{\mathbb N}
\newcommand{\C}{\mathbb C}

\newcommand{\pu}[1]{\mathscr{P}_{#1,n}U}
\newcommand{\puw}[1]{\mathscr{P}_{#1,n}^wU}

\newcommand{\ddcn}[1]{\left(\ddc #1\right)^n}
\newcommand{\sgn}[1]{\sign\left(#1\right)}
\newcommand{\MA}{Monge-Amp\`{e}re }

\newtheorem{theorem}{Theorem}[section]
\newtheorem{conjecture}{Conjecture}[section]
\newtheorem{proposition}{Proposition}[section]
\newtheorem{lemma}{Lemma}[section]
\newtheorem*{openproblem*}{Open problem}

\theoremstyle{definition}

\newtheorem{corollary}[theorem]{Corollary}

\theoremstyle{remark}
\newtheorem{remark}[theorem]{Remark}

\newcommand{\rev}[1]{#1}

\title[]{A pluripotential theoretic framework for polynomial interpolation of vector-valued functions and differential forms}

\author{Ludovico Bruni Bruno}
\address{Dipartimento di Matematica \textquotedblleft\emph{Tullio Levi-Civita}\textquotedblright, University of Padova, Via Trieste 63, Padova, Italia}
\email[L. Bruni]{bruni@math.unipd.it}

\author{Federico Piazzon}
\email[F. Piazzon]{fpiazzon@math.unipd.it}

\begin{document}
\begin{abstract}
We consider the problem of uniform interpolation of functions with values in a complex inner product space of finite dimension. This problem can be casted within a modified weighted pluripotential theoretic framework. Indeed, in the proposed modification a vector valued weight is considered, allowing to partially extend the main asymptotic results holding for  interpolation of scalar valued functions to the case of vector valued ones. As motivating example and main application we specialize our results to interpolation of differential forms by differential forms with polynomial coefficients. 
\end{abstract}
\maketitle
\begin{center}
\emph{The present work is dedicated to Len Bos in the occasion of his retirement}
\end{center}
\section{Introduction}

\subsection{Weighted pluripotential theory and polynomial approximation of functions}\label{sec:weightppt}
During the last few decades it \rev{became} clear that weighted pluripotential theory \cite{BlLe03} offers the correct \rev{framework} to understand asymptotic features of optimal polynomial interpolation and approximation of functions. This is demonstrated by a number of works concerning, e.g., polynomial inequalities \cite{BaBi14,BaBi13,Pi17,BlLePiWi15}, approximation  schemes \cite{Le06,BlBoCaLe12}, orthogonal polynomials \cite{Bl97,Pi19}, zeroes of random polynomials and random arrays \cite{BeR19,BaBlLeLu19,ZeZe10}, and experimental design \cite{BlBoLeWa10}, that relate pluripotential theory and approximation theory, as much as the number of numerical schemes derived from heuristics closely related to pluripotential theoretic results \cite{BoPiVi20,BoCaLeSoVi11,BoDeSoVi11}. 

In its essence, pluripotential theory (see \cite{KlM91} for an extensive treatment of the subject) is the study plurisubharmonic functions and of the complex \MA operator $\ddcn{\cdot}$. Plurisubharmonicity,  a property playing a pivotal role in whole complex analysis, is defined as the combination of upper semicontinuity (upc, for short), and subharmonicity along any one-dimensional affine subvariety of the considered domain. The complex \MA operator acting on a twice differentiable function $u\in \mathcal C^2(\Omega)$ of a domain $\Omega\subset \C^n$ is a constant multiple of the determinant of the complex Hessian matrix, i.e., $\det \partial\bar\partial u$, \rev{while it has been extended as} positive measure valued operator acting  on any locally bounded plurisubharmonic function $u\in \psh(\Omega)\cap L^\infty(\Omega)$; \cite{BeTa82}.

Weighted pluripotential theory is a generalization of pluripotential theory in which a positive weight function is considered. Historically (and particularly in the one-dimensional case), this generalization originated from the study of \rev{varying weight orthogonal polynomials}, and developed mainly for the purpose of dealing with regions which are not the complement of a compact set. Surprisingly, the weighted theory turned out to be the right formalism for solving long standing open problems steaming from the unweighted case; \rev{in particular, the well known problem of finding the asymptotic behavior of the Fekete points.}

Here we recall some definitions and results that will come into the play of our construction. Let $K\subset \C^n$ be a compact set and $Q:K\rightarrow \R$ be (for simplicity) a continuous function. We will denote by $w:K\rightarrow \R_+$ the function $\exp(-Q)$; \rev{$Q$ and $ w $ will be both referred to as weight functions, the context indicating which one is meant.} 
We can form the upper envelope
$$V_{K,Q}(z):=\sup\{u\in \mathcal L(\C^n): u\leq Q\text{ on }K\},$$
where $\mathcal L(\C^n)$ is the Lelong class of plurisubharmonic functions of at most logarithmic growth, i.e., $\mathcal L(\C^n):=\{u\in \psh(\C^n):u-\log\|z\|<C\text{  as }\|z\|\to +\infty\}.$ The function $V_{K,Q}$ is not in general plurisubharmonic because it may fail to be upper semicontinuous 
.  It is then considered its upper semicontinuous \rev{regularization}
$$V_{K,Q}^*(z):=\limsup_{\zeta\to z}V_{K,Q}(\zeta).$$
Two opposite scenarios may occour: either $V_{K,Q}^*$ remains locally bounded in $\C^n$, and in such a case it is also plurisubharmonic, or not. In \rev{the} latter case $K$ is termed \emph{pluripolar} because, \rev{roughly} speaking, $K$ is too small from the pluripotential theoretic point of view. Since in such a case the theory that we are going to describe does not apply,  \emph{we will always assume in what follows that $K$ is not pluripolar}. It is a deep result that, provided $K$ is not pluripolar, the function $V_{K,Q}^*$ is a \emph{maximal plurisubharmonic function}. Namely, it satisfies the homogeneous complex \MA $\ddcn{V_{K,Q}^*}=0$ equation in $\C^n\setminus K$. Indeed $(2\pi)^{-n}\ddcn{V_{K,Q}^*}$ is a probability measure supported in $K$, customarily denoted by $\mu_{K,Q}$ and termed \emph{weighted equilibrium measure} of $K$. 

If we aim for uniform convergence of polynomial interpolation of degree $r$ to a given continuous function (holomorphic in the interior of $K$) as the degree $r\to +\infty$, in principle we may try to maximize the modulus of the determinant of the Vandermonde relative to a given basis of $\mathscr P_{r,n}\C$, the space of complex polynomials of $n$ variables with degree at most $r$. This would lead to an approximation scheme based on a sequence of \rev{Lagrange interpolation} operators $I_r:\mathcal C^0(K)\rightarrow \mathscr P_{r,n}\C$ with slowly increasing norms $\|I_r\|\leq \ddim \mathscr P_{r,n}\C=:m_r.$ Triangular arrays $\{x_{1,r},\dots,x_{m_r,r}\}$, $m_r:=\ddim\mathscr P_{r,n}$, of interpolation points constrained on a compact set $K\subset \C^n$ and constructed by solving such a maximization procedure are known as \emph{Fekete points}. \rev{Zaharjuta \cite{Za75,Za12} showed that the sequence of appropriate powers of the maximized moduli of Vandermonde determinants has indeed a limit, called \emph{transfinite diameter} of the set $K$, and custumarily denoted by $\delta(K).$ A similar game can be played with varying weight polynomials, i.e., function of the form $p(x)w^r(x)$, where $p\in \mathscr P_{r,n}\C.$ Under mild assumptions on the weight function $w$ it is possible \cite{BlLe10} to define the \emph{weighted transfinite diameter} of $K$ by setting $\delta^w(K):=\lim_r \delta^{w,r}(K)$ (existence of the limit is stated in \cite[Prop. 2.7]{BlLe10}), where
\begin{equation}
	\delta^{w,r} (K) := \left[ \max_{\boldsymbol x\in K^{m_r}} \Big| \Vdm(x_1,\dots,x_{m_r}) \Big|w(x_1)^r\cdots w^r(x_{m_r}) \right]^{1/(\ell_r)}.
\end{equation}
Here $\Vdm$ denotes the standard Vandermonde determinant $\det [x_i^\alpha(j)]_{i,j=1,\dots, m_r}$, $N=\ddim \mathscr P_{r,n}$, and $\ell_r:=\sum_{j=1}^r[j(\ddim \mathscr P_{j,n}\C-\ddim \mathscr P_{j-1,n}\C)].$ }

\rev{Until the 60s, the interplay of all this quantities was well understood only in the context of $n=1$, but nothing was explained for the case $n>1$.} The celebrated work of Berman, Boucksom, and Nystrom \cite{BeBo10,BeBoNy11} detailed the whole picture by filling the gaps in the puzzle. Here we briefly recall only the part of their results that we will use later on. The first result in \cite{BeBo10} concerns the asymptotics of logarithmic ratios between the Haar volumes of uniform unit balls in the space of weighted polynomials with respect to two different normalizations, i.e. weights $w_1$, $w_2$ and compact sets $K_1$, $K_2$. Indeed, this asymptotic is $\mathcal E(V_{K_1,Q_1}^*,V_{K_2,Q_2}^*)$, the so-called relative \MA energy of $V_{K,Q_1}^*$ with respect to $V_{K,Q_2}^*$, where $Q_i:=-\log w_i.$ Notice that in the case of $K_2$ being the standard polydisk and $w_2\equiv 1$, the ball volume ratio tends to $\log \delta^{w_1}(K_1)$, as the degree of \rev{polynomials} \rev{tends to infinity}, while the weighted extremal function of the polydisk is $\max_j \log^+\|z_j\|$. So one has
\begin{equation}\label{TD-energy}
 - \log \delta^w(K)= \frac{1}{n(2 \pi)^n}\mathcal E(V_{K,Q}^*,\max_j \log^+\|z_j\|).
 \end{equation} 
\rev{This result is expressed by means of uniform norm unit ball volume.} An analog result is then obtained for $L^2_\mu$-balls, whenever $(K,w,\mu)$ satisfies an asymptotic comparability of $L^\infty(K)$ and $L^2_\mu$ norms of weighted polynomials. This hypothesis is known  in the literature as Bernstein Markov property \cite{BlLePiWi15}, and that can be thought as a particular instance of a Nikolski Inequallity. Namely, if we assume  that
\begin{equation}\label{bm-original}
\limsup_{r\to +\infty}\sup_{p\in \mathscr P_{r,n}\C\setminus\{0\}}\left(\frac{\|p w_i^r\|_{K}}{\|p w_i^r\|_{L^2_\mu}}\right)^{1/r}\leq 1,\; i=1,2,
\end{equation}
then the ball volume logarithmic ratio (which is nothing but a multiple of the logarithm of the determinant of a Gram matrix) tends to the aforementioned relative energy. Due to the orthonormality of the monomials on the boundary of the polydisk, one obtains 
\begin{equation}\label{gramasymptotic}
\lim_r \frac{n+1}{2 n r \ddim \mathscr P_{r,n}\C}\log \det G_r^{w}= - \frac{1}{n(2 \pi)^n} \mathcal E(V_{K,Q}^*,\max_j \log^+\|z_j\|)=\log \delta^w(K).
\end{equation}
In the above equation $G_r^{w}$ is the Gram matrix of the scalar product of $L^2_\mu$ written in the basis of weighted monomials of degree at most $r$.

As second step in the construction of \cite{BeBo10,BeBoNy11}, the differentiability of the function 
$$t\mapsto  E(t):=\mathcal E(V_{K,Q+tu}^*,\max_j \log^+\|z_j\|)$$
is proven, and the derivative computed at zero, which gives 
\begin{equation} \label{eq:derEner}
    E'(0)= (n+1) \int_Ku\ddcn{V_{K,Q}^*}.
\end{equation}
Rather surprisingly, a smart use of the concavity of the considered functions \cite[Lemma 6.6]{BeBo10} allows to bring the derivativation with respect to $t$  inside the limit in \eqref{gramasymptotic}. Thus, by direct computations one obtains the \emph{strong Bergman asymptotics}
\begin{equation}\label{sba-original}
\int_K u(x)\frac{B_{w,r}(x,\mu)}N d\mu(x)\to \frac{1}{(2 \pi)^n}\int_K u \ddcn{V_{K,Q}^*},\;\forall u\in \mathscr C^0(K,\R),
\end{equation}
where $B_{w,r}(x;\mu):=\sum_{i=1}^{\ddim \mathscr P_{r,n}\C}|q_i(x)|^2 w^{2r}(x)$ is the diagonal of the reproducing kernel of weighted polynomials of degree at most $r$, i.e., the $q_i$'s form an $L^2_\mu$-orthonormal basis of $\mathscr P_{r,n}\C$.

\rev{Finally, the convergence of (weighted) Fekete points to the (weighted) equilibrium measure $(2\pi)^{-n}\ddcn{V_{K,Q}^*}$} is obtained in \cite{BeBo10} as a particular case of \eqref{sba-original} by noticing that the associated probability measures $\mu^{(r)}$ satisfy the Bernstein Markov property \eqref{bm-original}, and that for any interpolation array one has that the corresponding function $B_{r,w}(\cdot,\mu^{(r)})=N$ (at any interpolation point and hence) $\mu^{(r)}$-a.e, thus
\begin{equation}
\label{Feketeasympt}
\mu^{(r)}\rightharpoonup^*\mu_{K,Q}, \text{ as }r\to+\infty\,,
\end{equation} 
where $\rightharpoonup^*$ denotes weak$^*$ convergence of measures.

\subsection{Our study}
In the present work we focus on polynomial interpolation of functions with values in a complex Hermitian space of dimension $s$, with application to the case of differential forms. In these contexts the point-wise evaluation functionals take values in $\C^s$ instead of simply in $\C$. Therefore, if we aim at studying vector-valued polynomial interpolation,  we first need to construct a \emph{vectorized} pluripotential theoretic framework in which the aforementioned usual definitions, techniques, and results have natural extensions.
We carry out this generalization independently from the application to the space of differential forms, which is detailed in Section \ref{sect:application}.
\subsubsection{Notation and preliminaries} Let $\left(U,(\cdot,\cdot)_U\right)$ be an $s$-dimensional Hermitian space over $\C$, and let $u_1,\ldots,u_s$ be an orthonormal basis of $U$. Let us denote by $U^\R$ the real vector space spanned by $u_1,\ldots,u_s$. We denote by $\pu{r}$ the space of polynomial functions from $\C^n$ to $U$, and, for any $K\subset \C^n$, we denote by $\pu{r}(K)$ the space of polynomial functions from $K$ to $U$. Using the isometric isomorphism of $U\cong \C^s$ induced by the choice of the basis $\{u_1,\ldots,u_s\}$, we can construct a basis for $\pu{r}\cong \mathscr P_{r,n}\C^s$ by considering
\begin{equation}\label{eqn:basisq}
q_j(x):=x^{\beta(j)}u_{s(j)}\longmapsto p_j(x):=x^{\beta(j)}e_{s(j)},\;j=1,\dots,s\times \ddim \mathscr P_{r,n} \C=:s\times m_r=: N.
\end{equation}
\rev{Here we set $j=ls+r$, with $r$ the reminder of the integer division of $j$ by $s$. We introduce $s(j)$ to be $r$ whenever $r\neq 0$ and to take value $s$ if $r=0$. Then we let $x^\alpha=\prod_{i=1}^nx_i^{\alpha_i}$, $e_k$ is the $k$-th element of the canonical basis of $\C^s$, $\beta(j)$ is the $l$-th element of $\N^n$ with respect to the graded lexicographical} ordering. Note that our choice of coordinates and scalar products implies in particular that
$${(p_j(x),p_k(x))}_{\C^s}={(q_j(x),q_k(x))}_{U}=\bar x^{\beta(k)}x^{\beta(j)}(u_{s(j)},u_{s(k)})_U=\bar x^{\beta(k)}x^{\beta(j)}\delta_{s(j),s(k)}\,.$$
The canonical coordinates in $\C^s$ yield the coordinate-wise multiplication $\odot:\C^s\times \C^s\rightarrow \C^s$, with $(y\odot z)_j:=y_jz_j$, for any $j=1,\dots,s.$ Again resting upon the identification $U \cong \C^s$, we can define, by a slight abuse of notation, a binary operation \rev{$\odot: \mathscr C^0(K,U) \times \mathscr C^0(K,U)\rightarrow \mathscr C^0(K,U)$ by setting $a(x)\odot b(x):=\sum_{i=1}^s a_i(x) b_i(x)u_i(x)$, for any $a(x):=\sum_{i=1}^sa_i(x) u_i,$ and $b(x):=\sum_{j=1}^sb_j(x) u_i$ (note that the continuity of involved functions has been assumed only for simplicity).}

Given an $\R^s$ valued continuous weight function $w:K\rightarrow \R_+^s,$ being  $K\subset \C^n$ compact, we denote by $\puw{r}(K)$ the space of varying weight polynomial functions from $K$ to $U$, i.e., functions of the form $q\odot w^r(x):=q(x)\odot(\sum_{j=1}^s w_j(x)^ru_j),$ with $q\in \pu{r}$. Note that in the set up of our notation we are implicitely identifying any such continuous $\R^s$-valued function $w$ with an element of $\mathscr C^0(K,U^\R).$

Clearly one has $\mathscr P_{r,n}^w\C^s\cong \puw{r}$, so that the two spaces can be identified. A basis for $\mathscr P_{r,n}^w\C^s$ is indeed given by
$$\left\{p_j(x)\odot w^r(x), j=1,\dots,N\right\}:=\left\{p_j(x)\odot\overbrace{w(x)\odot\cdots\odot w(x)}^{\text{r times}}, j=1,\dots,N\right\},$$
and the isometrically isomorphic basis for $\puw{r}$ is constructed according to \eqref{eqn:basisq}.

We will think to $\pu{r}(K)$ and $\puw{r}(K)$ as closed subspaces of the Banach space $\left(\mathscr C^0(K,U),\|\cdot\|_{K,U}\right)$ of continuous $U$-valued functions, where 
$$\|\omega\|_{K,U}:=\max_{x\in K}\|\omega(x)\|_U.$$ 
Accordingly, we introduce the Banach space 
$$\mathcal M(K,U):=( \mathscr C^0(K,U))^*$$
of $U$-valued measures on $K$. Due to Riesz' Representation Theorem, each element $T$ of $\mathcal M(K,U)$ may be written as
\begin{equation}\label{eq:Tdef}
T(q):=\int_K {(v(x),q(x))}_U d\mu(x),\;\forall q\in \mathscr C^0(K,U),
\end{equation}
where $\mu$ is a Borel measure on $K$, and $v:K\rightarrow U$ is an $L^2_\mu$ function such that $\|v(x)\|_U=1$ $\mu$-a.e. Apart from the strong convergence in $\mathcal M(K,U)$, defined by the operator norm $\|T\|:=\sup\{|T(q)|,q\in \mathscr C^0(K,U), \|q\|_{K,U}=1\}$, we will also consider weak$^*$ convergence, customarily denoted by $\rightharpoonup^*$. For the reader's convenience we recall that $\{T_j\}$ converges to $T$ weakly$^*$, and we write $T_j\rightharpoonup^* T$, if and only if $\lim_j T_j(q)=T(q)$ for any $q\in \mathscr C^0(K,U).$
\subsubsection{Main results}
In Section \ref{sec:WTD} we extend to the above described vectorial framework the concept of weighted transfinite diameter $\delta^w(K)$ of a compact set $K\subset \C^n$ as the asymptotics of maximal generalized Vandermonde determinants. This new quantity, subsequently denoted by $\delta^w(K,U)$, is characterized in Theorem \ref{thm:MR1} by considering an appropriate geometric mean of the usual transfinite diameters relative to each $u_i$ direction in $\puw{r}$. Since the logarithm function transforms geometric means into arithmetic ones, we can apply all the machinery developed in the scalar case in a component-wise fashion, see Eq. \eqref{TD-energy}. In particular, the derivative of the function $(-\log w_1,\dots,-\log w_s)\mapsto -\log \delta^w(K,U)$ can be related to integral functionals depending on the equilibrium measure of $K$ with respect to each $w_i$, due to \eqref{eq:derEner}.

The Bernstein Markov property is used in the scalar case \rev{to replace} the \textquotedblleft$L^\infty$\textquotedblright -maximization procedure present in the definition of $\delta^w(K)$ by a more manageable \textquotedblleft$L^2$\textquotedblright -maximization. This allows in particular to compare the asymptotics of Gram determinants of so-called Bernstein Markov measures with $\delta^w(K)$, and -- passing to derivatives of the objective with respect to the logarithm of the weight --, to obtain the Bergman asymptotic \eqref{sba-original} and  the convergence of Fekete points \eqref{Feketeasympt}. In Section \ref{sec:BM-Gram} we retrace this construction in the vector valued case. In particular, working with Bernstein Markov vector measures (see Eqn.s \eqref{BM-prop} and \eqref{WBM-prop} later on), we show that appropriate roots of $\log$-determinants of the Gram matrix of $\puw{r}$ with respect to the scalar product induced by a vector Bernstein Markov measure tend to $\log\delta^w(K,U)$; see Proposition \ref{prop:MR2}. 

\rev{In Section \ref{sec:Fekete}, specifically in Proposition \ref{thm:MR4}, we obtain appropriate extensions of convergence of Fekete points. Note that, due to the vectorized framework we work in, usual Fekete points are replaced by higher dimensional counterparts, which we identify with vector point-masses. The higher geometrical complexity of our setting carries some additional difficulties in proving a vector counterpart of the strong Bergman asymptotics \eqref{sba-original}, which is in fact only conjectured in this work (see Conjecture \ref{thm:MR3}).}

\subsubsection{Polynomial interpolation of differential forms}    

A case of particular interest is the one of $ U^\R = \Lambda^k_\R $, i.e. the space of real alternating $k$-covectors spanned by $dx^{\alpha}:=dx^{\alpha_1}\wedge \dots\wedge dx^{\alpha_k}$ with $\alpha_1<\alpha_2<\dots<\alpha_k$. In this case, the space $ \mathscr{P}_{r,n} \Lambda^k_\R \subset \mathscr{C}^0 \Lambda^k_\R $ is that of \emph{real polynomial differential forms}, namely real polynomial sections of the $k$-th exterior power of the cotangent bundle of $ \R^n $; see \cite{Warner}. Although in this application we will always assume $K\subset\R^n$, in our construction we will look at $\Lambda^k_\R$ as a real subspace of the complex Hermitian space
$$\Lambda^k_\C:={\Span}_\C\{dx^{\alpha}:|\alpha|=k,\,\alpha\text{ is increasing}\},$$
endowed by the Euclidean product. We remark that this embedding is very natural from the perspective of approximation theory, even if it might sound odd from the point of view of differential geometry.
 
 Functions with values in $ \Lambda^k_\R $ are generally called \emph{tensors} and measure \emph{how} a reference tensor field (for instance, a vector field defined on $ K $ when $ k = 1 $) is modified at each point of $ \R^n $; see \cite{AMRBook}. \rev{Important examples} of this kind are the elasticity tensor, the stress tensor or the Faraday $2$-form in electromagnetism \cite{Flanders}. These quantities describe physical objects, so that a correct understanding of their polynomial counterpart is essential in their approximation. The \rev{use} of polynomial differential forms ranges from interpolation theory \cite{BBZ22} to finite element methods \cite{AFW} and structure preserving methods \cite{Hiemstra}.

To interpolate in $ \mathscr P_{r,n} \Lambda^k_\C $, a set of linear functionals $ T_i : \mathscr{C}^0 \Lambda^k \to \C $, with $ i = 1, \ldots, N := \dim \mathscr P_{r,n} \Lambda^k_\C $, is needed. These objects are named \emph{currents} \cite{Federer}, and in the simplicial context they are usually represented by means of moments \cite{Nedelec86} or weights \cite{RapettiBossavit}. When $ \{T_i\}_{i=1}^N $ forms a basis for $ \left( \mathscr P_{r,n} \Lambda^k_\C \right)^* $, the dual space of $ \mathscr P_{r,n} \Lambda^k_\C $, the set $ \{T_i\}_{i=1}^N $ is said to be unisolvent for $ \mathscr P_{r,n} \Lambda^k_\C $. Hence, we may represent any $ \omega \in \mathscr P_{r,n} \Lambda^k_\C $ as
\begin{equation} \label{eq:interpolationoperator}
    \omega = \sum_{i=1}^N T_i (\omega) \omega_i,
\end{equation}
being $ \{\omega_i\}_{i=1}^N $ the Lagrangian basis for $ \mathscr P_{r,n} \Lambda^k_\C $, namely that satisfying $ T_i (\omega_j) = \delta_{i,j} $. Note that in the above formula the currents $ T_i (\omega) $ may be thought of as measurements of the physical quantity $ \omega $. Given a compact real set $K$ (determining for $ \mathscr P_{r,n} \Lambda^k_\C $ for any $r\in \N$), we will always consider the uniform norm on $\mathscr{C}^0(K, \Lambda^k_\C)$, that is 
$$\|\omega\|_{\mathscr{C}^0(K, \Lambda^k_\C)}:=\sup_{x\in K}\|\omega(x)\|_{\Lambda^k_\C}.$$ 
Therefore the norm $ \Vert \Pi \Vert_{\mathrm{op}} $ of the interpolation operator $ \Pi : \mathscr{C}^0 (K,\Lambda^k_\C) \to \mathscr P_{r,n} \Lambda^k_\C $ specified by \eqref{eq:interpolationoperator}, \rev{i.e. $ \Pi \omega := \sum_{i=1}^N T_i (\omega) \omega_i $ for $ \omega \in \mathcal{C}^0 (\Lambda^k_\C) $}, is named \emph{Lebesgue constant} and quantifies the stability of the interpolation via the Lebesgue inequality
$$ \rev{\Vert \theta - \Pi \theta \Vert_{\mathscr{C}^0 (K\Lambda^k_\C)} \leq \left( 1 + \Vert \Pi \Vert_{\mathrm{op}} \right) \min_{\omega \in \mathscr P_{r,n} \Lambda^k_\C}\Vert \omega - \theta \Vert_{\mathscr{C}^0(K, \Lambda^k_\C)}} .$$
In the \rev{search of} currents that are well-suited for interpolation, Fekete problems play an essential role. By Fekete problem we mean the identification of a collection of unisolvent currents $ \{ T_i \}_{i=1}^N $ such that the (modulus of the) determinant of the \emph{Vandermonde matrix}
\begin{equation} \label{eq:VdmMatrix}
    V_{i,j} = T_{i} (q_j),
\end{equation}
\rev{with respect to a given (and thus any) basis $q_1,\dots,q_N$ of $ \mathscr P_{r,n} \Lambda^k_\C $} is maximal. When $ k = 0 $, i.e. in the case of functions, the solution of this problem has been investigated for a long time. \rev{Note that any uniform probability measure supported at \emph{Fekete points} is a solution of such a maximization problem.} When $ k > 0 $ the situation becomes more involved and only a recent result offers an answer in the particular setting of weights \cite{BE24}.

\rev{\subsection*{Acknowledgements}
The work of Professor Bos, together with his helpfulness, had and is still having an enourmous impact on the mathematical developement of the Constructive Approximation and Applications research group of Padova and Verona, the research team where the second author grown up. For this reason he feels indebted and grateful with Prof. Len Bos. His ideas and techniques on unisolvence have been precious in the development of several works of the first author.

Both authors deeply thanks the two anonymous referees for their valuable and precise comments that allowed the authors to improve the present manuscript.}

\section{Weighted transfinite diameter}\label{sec:WTD}
This section is devoted to introducing a notion of \rev{$r$-th} weighted transfinite diameter of a compact set $K\subset\C^n$, \rev{with respect to the weight $ w $ and the space $ U $}, which is well suited for the study of varying weight polynomial interpolation of vector valued holomorphic functions.

Let $ m_r := \binom{n+r}{n} $, so that $ \dim \mathscr{P}_{r,n} \C = m_r $. We define the $r$-th weighted diameter of a set $K$ with respect to the space $U$ as
\begin{equation} \label{eq:defweighteddiam}
	\delta^{w,r} (K,U) := \left[ \max_{\substack{T_1 \in \mathcal M(K,U) \\ \Vert T_1 \Vert = 1}} \ldots \max_{\substack{T_N \in \mathcal M(K,U)\\ \Vert T_N \Vert = 1}} \Big| \det [T_i \left( q_j \odot w^r \right)]_{i,j=1,\dots,N} \Big| \right]^{1/(s\cdot\ell_r)},
\end{equation} 
where $ \ell_{r} := \sum_{k=1}^r k \left( m_k - m_{k-1} \right),$ and the $q_j$'s have been defined in \eqref{eqn:basisq}.

As announced, we are able to prove that the limit of the $r$-th diameters does exist, and it equals the geometric mean of the classical weighted transfinite diameters obtained by considering each component of the vectorial weight $w$.
\begin{theorem}\label{thm:MR1}
	Let $ K \in \C^n $ be a compact set, $U$ as above and $ w_l \in \mathscr{C}^0 (K, ]0,+\infty[)$ for each $l=1,\dots,s=\ddim U$. Then
	\begin{equation} \label{eq:transfinitediameter}
		\delta^w(K,U):=\lim_{r\to \infty} \delta^{w,r} (K,U) = \left( \prod_{l=1}^s \delta^{w_l} (K) \right)^{1/s}, 
	\end{equation}
	where existence of the limit is part of the statement.
\end{theorem}

\begin{proof}
	The result \rev{may} be proved by generalizing to the vectorial case the argument of V. Zaharjuta adopted in \cite{Za75,Za12}. \rev{Instead, we propose a proof that uses algebraic manipulations that essentially allow us to go back to the scalar case.}
	
	Let $ T_1, \ldots, \textcolor{red}{T_N} \in \mathcal M(K,U) $, with $ \Vert T_i \Vert = 1 $ for each $ i $, \rev{ be such that $ \left\vert \det T_i \left( p_j \odot w \right) \right\vert^2 = (\delta^{w,r} (K,U))^{2 s\cdot \ell_r} $}. We then have, \rev{denoting by $\Sigma_N$ the set of permutation of $N=s m_r$ elements,}
	\begin{align*}
		& (\delta^{w,r} (K))^{2s  \ell_r}= \left\vert \det T_i \left( q_j \odot w \right) \right\vert^2  \\
		= &   \left\vert \bigotimes_{i=1}^N T_i\left( \sum_{\sigma \in \Sigma_N} \sgn \sigma \bigotimes_{j=1}^N \left( q_{\sigma(j)} \odot w^r \right) (x_1, \ldots, x_N)\right) \right\vert^2  \\
		\leq & \left\Vert \bigotimes_{i=1}^N T_i \right\Vert_{[\mathcal M(K,U)]^N}^2 \max_{\boldsymbol{x}\in K^N} \left\Vert \sum_{\sigma \in \Sigma_N} \sgn \sigma \bigotimes_{j=1}^N \left( q_{\sigma(j)} \odot w^r \right) (x_1, \ldots, x_N) \right\Vert^2_{U^N}  \\
	 	= & \max_{\boldsymbol{x}\in K^N} \sum_{\sigma \in \Sigma_N} \sum_{\tau \in \Sigma_N}  \sgn \sigma \sgn \tau \left( \bigotimes_{j=1}^N \left( q_{\sigma(j)} \odot w^r \right) (x_1, \ldots, x_N), \bigotimes_{j=1}^N  \left(q_{\tau(j)} \odot w^r \right) (x_1, \ldots, x_N) \right)_{U^N} \\
	 	 = & \max_{\boldsymbol{x}\in K^N} \sum_{\sigma \in \Sigma_N} \sum_{\tau \in \Sigma_N}  \sgn \sigma \sgn \tau \prod_{j=1}^N \Big( (p_{\sigma(j)} \odot w^r) (x_j),  (p_{\tau(j)} \odot w^r )(x_j) \Big)_{\C^s} .
	\end{align*}
	Now, notice that $ \left(\left( p_{\sigma(j)} \odot w^r \right) (x_j),  \Big(p_{\tau(j)} \odot w^r \right) (x_j) \Big)_{\C^s} = 0 $ if $ s(\sigma(j)) \neq s(\tau(j)) $ for some $ j \in \{1, \ldots, N \} $, i.e., the only $ \tau\in \Sigma_N $ contributing to the sum are the one for  which $s(\tau(j))=s(\sigma(j))$ for all $j=1,\dots,N$ and for which the set $\{\beta(\tau(j)), j:s(\tau(j))=l\}$ can be obtained by a permutation of the set $\{\beta(\sigma(j)), j:s(\sigma(j))=l\}.$ Namely, we consider only $\tau$ of the form $ \tau = \widetilde{\tau} \circ \sigma $ with $\widetilde{\tau} = \widetilde{\tau}_{l} \circ \ldots \widetilde{\tau}_{s} $ and $ \widetilde{\tau}_{l} = \mathbb{I}_{s} \otimes \eta_l, \eta_l \in \Sigma_m $ for all $l=1,\dots, s$. Thus, introducing the notation $ J (\sigma,l) := \{ j \in \{1, \ldots, N\} : s(\sigma(j)) = l \} $, we have
	\begin{align*}
		& (\delta^{w,r} (K))^{2 s \ell_r} \\
		= & \max_{\boldsymbol{x}\in K^N} \sum_{\sigma \in \Sigma_N} \sum_{\eta_1 \in \Sigma_m} \ldots \sum_{\eta_s \in \Sigma_m} (\sgn \sigma)^2 \prod_{l = 1 }^s \sgn {\eta_l} \prod_{j \in J(\sigma,l)} \bar x_j^{\beta(\sigma(j))} x_j^{\beta(\eta_l (\sigma (j)))} w_l^{2r} (x_j) \\
		= & \max_{\boldsymbol{x}\in K^N} \sum_{\sigma \in \Sigma_N} \prod_{l=1}^s \sum_{\eta \in \Sigma_m}\sgn {\eta} \prod_{j \in J(\sigma,l)} \bar x_j^{\beta(\sigma(j))} x_j^{\beta(\eta(\sigma (j)))} w_{l}^{2r} (x_j)\\
		\rev{\leq}& \sum_{\sigma \in \Sigma_N} \prod_{l=1}^s \max_{\boldsymbol x_{J(\sigma,l)}:=(x_{J(\sigma,l)_1},\dots,x_{J(\sigma,l)_{m_r}})\in K^{m_r}}\left\vert \det V(\sigma,l,\boldsymbol x_{J(\sigma,l)}) \right\vert^2\\
		 =& \sum_{\sigma \in \Sigma_N} \prod_{l=1}^s \max_{\boldsymbol z\in K^{m_r}}\left\vert \Vdm(z_1,\dots,z_{m_r}) \right\vert^2w_l(z_1)^{2r}\dots w_l(z_{m_r})^{2r}\\
		\rev{=} & N! \prod_{l = 1}^s \left( \delta^{w_{l},r} (K) \right)^{2 s \ell_r },
	\end{align*}
	where $\Vdm(z_1,\dots,z_{m_r})$ is the classical Vandermonde determinant with respect to the monomial basis, $ (V(\sigma,l,\boldsymbol{x}_{J(\sigma,l)}))_{h,k} := x_{(J(\sigma,l))_h}^{\beta(\sigma(k))} w_{l}(x_{(J(\sigma,l))_h})^{r}$, and we used that 
	\begin{align*}
	 &\sum_{\eta \in \Sigma_m}\sgn {\eta} \prod_{j \in J(\sigma,l)} \bar{x}_j^{\beta(\sigma(j))} x_j^{\beta(\eta (\sigma (j)))} w_{l}^{2r} (x_j) = \det [V(\sigma,l,\boldsymbol{x}_{J(\sigma,l)})^H V(\sigma,l,\boldsymbol{x}_{J(\sigma,l)})]\\ =& |\det V(\sigma,l,\boldsymbol{x}_{J(\sigma,l)}) |^2	  = \vert \Vdm(x_{J(\sigma,l)_1},\dots,x_{J(\sigma,l)_{m_r}}) \vert^2 w_l^{2r}(x_{J(\sigma,l)_1})\cdots w_l^{2r}(x_{J(\sigma,l)_{m_r}}),
	 \end{align*}
since, for any $\sigma \in \Sigma_N$ and any $l=1,\dots,s$,
	$$\{\beta(\sigma(j)), j\in J(\sigma,l)\}=\{\beta\in \N^n: |\beta|\leq r\}.$$
	Therefore, taking the limit superior as $r$ tends to $+\infty$, we get
	$$ \limsup_r \delta^{w,r}(K,U) \leq \limsup_r \left( N! \right)^{1/(2s \ell_r)} \cdot \limsup_r \left( \prod_{l=1}^s \delta^{w_l, r} (K) \right)^{1/s} = \left( \prod_{l=1} ^s\delta^{w_l} (K) \right)^{1/s} . $$
	This shows the upperbound.

	Conversely, let $ x_{1,l}^{(r)}, \ldots, x_{m_r,l}^{(r)} $ be weighted Fekete points of degree $ r $ for the weight $w_l$, and let
	$$ y = \underbrace{x_{1,1}^{(r)}, \ldots, x_{1,s}^{(r)}}_{s \text{ elements}}, \ldots, \underbrace{x_{m_r,1}^{(r)}, \ldots, x_{m_r,s}^{(r)}}_{s \text{ elements}} .$$
	We consider the $U$-valued vector measures $T_i^{(r)}$'s defined by setting $ T_i^{(r)} (\omega) :=(\omega(y_i),u_{s(i)})_U $ and we observe that
	\begin{align*}
		\left\vert \Vdm (T_1^{(r)}, \ldots, \textcolor{red}{T_N}^{(r)}, q_1 \odot w, \ldots, q_N \odot w) \right\vert^{1/\ell_r} =& \left(\prod_{l=1}^s \left\vert \Vdm(x_{1,l},\dots,x_{m_r,l}) \right\vert^2w_l(x_{1,l})^{2r}\dots w_l(x_{m_r,l})^{2r}\right)^{1/(s \ell_r)}\\
		 =& \left( \prod_{l=1}^s \delta^{w_l, r} (K) \right)^{1/s}.
	\end{align*}
	On the other hand, the vector measures $ T_1^{(r)}, \ldots, \textcolor{red}{T_N}^{(r)} $ define a competitor in the maximization that defines $ \delta^{w,r} (K,U) $. Thus
	$$ \delta^{w,r} (K,U) \geq \left\vert \Vdm (T_1^{(r)}, \ldots, \textcolor{red}{T_N}^{(r)}, q_1 \odot w, \ldots, q_N \odot w) \right\vert^{1/\ell_r}= \left( \prod_{l=1}^s \delta^{w_l, r} (K) \right)^{1/s} .$$
	Taking the lim-inf as $ r \to \infty $, the proof is concluded.
\end{proof}

\section{Bernstein Markov Property, Gram determinants, and "free energy asymptotics"}\label{sec:BM-Gram}
\subsection{Bernstein Markov property in $\mathcal M(K,U)$}
Let $T\in\mathcal M(K,U)$ be represented by the Borel measure $\mu$ and the $\mu$-measurable function $v:K\rightarrow U$, see \eqref{eq:Tdef}. We can associate to this linear functional a rank-one semimetric $g$ on $U$ by setting
$$g_x(\omega,\theta):=\overline{(v(x),\omega(x))_U} (v(x), \theta(x))_U.$$
\rev{Note that we will often omit the dependence of $g$ on $x$ to simply our notation.} Such semimetric naturally defines (by means of integration with respect to $\mu$) a seminorm (induced by a semidefinite scalar product) on $\mu$-measurable functions $K\rightarrow U$ given by
$$\|\omega\|_{v,\mu}:=\left(\int_K g_x(\omega,\omega)d\mu(x)\right)^{1/2}.$$
Although $g$ is merely a semimetric with low rank, it may happen that $\|\cdot\|_{v,\mu}$ is indeed a norm on $\pu{r}$; in such a case we term the pair $(v,\mu)$ (or, equivalently, the linear functional $T$) $\pu{r}$-determining. Note that the easiest way to construct such a determining pair is to pick $x_1,\dots,x_M\in K$ \rev{(recall that here and throughout the paper we assume $K$ to be polynomial determining)} and $v_1,\dots,v_M\in U$, with $\|v_h\|_{U}=1$, such that, defining $T_h(\omega):=(\omega(x_h),v_h)_U$, the matrix $V:=[T_h(e_k)]_{h=1,\dots,M,k=1,\dots, N}$ has full column-rank. Indeed, in such a case one can set $\mu:=\sum_{h=1}^M \delta_{x_h}$ and $v(x_h)=v_h$, and get $g_{x_h}(\omega,\theta)=\overline{T_h(\omega)}T_h(\theta)$, so that $\|\omega\|_{v,\mu}^2=\sum_hT_h^2(\omega)$, being the associated scalar product on $\pu{r}$ represented by $V^HV.$ 

Since $\pu{r}$ is finite dimensional, it is clear that, for any triple $(K,v,\mu)$ with $(v,\mu)$ $\pu{r}$-determining, there exists a constant $C$ such that 
$$\|\omega\|_K:=\max_{x\in K}|\omega(x)|\leq C\|\omega\|_{v,\mu},\;\forall \omega \in \pu{r}.$$
However, in what follows we will always make an assumption on the asymptotic comparability of these two norms, as $r\to +\infty.$ Namely, we assume that $(K,v,\mu)$ (or $(K,v^{(r)},\mu^{(r)})$ sometimes) satisfies the \emph{Bernstein Markov property} \cite{BlLePiWi15}, that in the context of $\pu{r}$ reads as:
\begin{equation}\label{BM-prop}
\limsup_{r\to +\infty}M_r(K,v^{(r)},\mu^{(r)})^{1/r}:=\limsup_{r\to +\infty}\sup_{\omega\in \pu{r}\setminus\{0\}}\left(\frac{\|\omega\|_{K,U}}{\|\omega\|_{v^{(r)},\mu^{(r)}}}\right)^{1/r}\leq 1.
\end{equation}
The supremum in such a definition is indeed a maximum. In the scalar case $U=\C$ this is an immediate consequence of Parseval Identity, being $M_r$ the square root of the diagonal of the reproducing kernel, i.e., $M_r=\sqrt{B(x,\mu)}=\max_{x\in K}(\sum_{i=1}^{m_r}|b_i(x)|)^{1/2},$ with $\int_K \bar b_i b_j d\mu=\delta_{i,j}$ any orthonormal basis of $\mathscr P_{r,n}\C.$ In our framework, given an orthonormal basis $b_1,\dots,b_N$ of $\pu{r}$, we can form the Hermitian positive definite matrix $G(x)$, with $G_{i,j}(x):=(b_i(x),b_j(x))_U$. For any $x\in K$, we can interpret the quantity
$$\sup_{\omega\in \pu{r}\setminus\{0\}}\frac{\|\omega(x)\|_{U}^2}{\|\omega\|_{v,\mu}^2}=\max_{z\in \C^s}\frac{z^HG(x)z}{\|z\|^2}=:B(x,\mu,v)$$
as Rayleigh quotient of the matrix $G(x)$, so that $M_r(K,v,\mu)=\max_{x\in K}\sqrt{\lambda_{max}G(x)}$.

We need to extend the notion of Bernstein Markov  in $\mathcal M(K,U)$ property in two directions: we need to possibly include a \emph{varying weight} on the coefficients, and we need to consider the tensor product version of both the weighted and unweighted instances of the Bernstein Markov property.  We remark that in what follows all the \textquotedblleft weighted statements\textquotedblright\ strongly depend on the choice of orthonormal coordinates in $U$.  

A sequence $(v^{(r)},\mu^{(r)})$ of $\puw{r}$-determining vector measures is said to enjoy the weighted Bernstein Markov property with respect to the weight $w=(w_1,\dots,w_s)$ if
\begin{equation}\label{WBM-prop}
\limsup_{r\to +\infty}M_r(K,w,v^{(r)},\mu^{(r)})^{1/r}:=\limsup_{r\to +\infty}\sup_{\omega\in \pu{r}\setminus\{0\}}\left(\frac{\|\omega\|_{K,U,w^r}}{\|\omega\|_{w^r,v^{(r)},\mu^{(r)}}}\right)^{1/r}\leq 1,
\end{equation}
with $\|\omega\|_{K,U,w^r}:=\sup_{x\in K}\|\omega(x)\odot w^r(x)\|_U$ and $\|\omega\|_{w^r,v^{(r)},\mu^{(r)}}:=\|\omega\odot w^r\|_{v^{(r)},\mu^{(r)}}.$

\rev{\begin{remark}[BM vector measures asssociated to Fekete points]\label{FeketeareBM}
It is worth pointing out here a relevant way to construct a Bernstein Markov pair $(v^{(r)},\mu^{(r)})$ for a given polynomial determining compact set $K$ and a weight $w$. Let us pick, for any $l=1,2,\dots,s$, an asymptotically Fekete (for scalar polynomial interpolation) triangular array $(x_{h,r}^{(l)})_{h=1,\dots,m_r,r\in \N}$ of points of $K$ with respect to the weight $w_l$, i.e., assume 
$$\lim_r \Big|\det[(x_{h,r}^{(l)})^{j-1}w_l^r(x_{h,r}^{(l)})]_{h,j}\Big|^{1/l_r}=\delta^{w_l}(K),\;\forall l=1,\dots,s.$$
It is clear that, without loss of generality, we can assume $\{x_{1,r}^{(l_1)},\dots,x_{m_r,r}^{(l_1)}\}\cap \{x_{1,r}^{(l_2)},\dots,x_{m_r,r}^{(l_2)}\}=\emptyset$ for any $l_1\neq \l_2$.

Let us now introduce the probability measure $\mu^{(r)}:=\frac 1 s\sum_{l=1}^s\mu^{(r)}_l:=\frac 1 s\sum_{l=1}^s\frac 1{m_r}\sum_{h=1}^{m_r}\delta_{x_{h,r}^{(l)}}$, and the $\mu^{(r)}$-a.e. $U$-valued function $v^{(r)}(x)=u_i$, whenever $x=x_{h,r}^{(i)}.$ Using the scalar case one has that each measure $\mu^{(r)}_l$, $l=1,\dots s$, enjoys a weighted Bernstein Markow property on $K$ with respect to the weight $w_l$. But notice that the orthogonallity of $u_i$'s leads to
\small
$$\|\omega\|_{K,U,w^r}^2=\max_{x\in K}\sum_{i=1}^s|\omega_i(x)w_i^r(x)|^2\leq \sum_{i=1}^s\max_{x\in K}|\omega_i(x)w_i^r(x)|^2\leq \max_{l=1,\dots,s}M_{r,l}\sum_{i=1}^s\|\omega_i\|^2_{\mu_i^{(r)},w_i^r}=:M_r\|\omega\|_{v^{(r)},\mu^{(r)},w^r},$$ 
\normalsize
so that the weighted Bernstein Markov property of the quadruple $(K,w,v^{(r)},\mu^{(r)})$ immediately follows. 
\end{remark}}

Given a $\puw{r}$-determining quadruple $(K,w,v,\mu)$ one can define a semimetric on \rev{$U^t$, $t\in \N$} which induces a norm on tensor products of $\puw{r}$ simply setting $g^{\otimes}_{\boldsymbol x}(\otimes_{h=1}^t z_i,\otimes_{h=1}^t y_i)=\prod_{i=1}^tg_{x_i}(z_i,y_i)$, and 
\begin{align*}
\|\otimes_{i=1}^t \omega_i\|_{v,\mu}:=&\int_K\dots\int_Kg^{\otimes}_{\boldsymbol x}(\otimes_{i=1}^t \omega_i(x_i),\otimes_{i=1}^t \omega_i(x_i))d\mu(x_1)\dots d\mu(x_t)\\
=&\int_K\dots\int_K\prod_{i=1}^t g_{x_i}(\omega_i(x_i),\omega_i(x_i))d\mu(x_1)\dots d\mu(x_t),
\end{align*}
and similarly for the weighted case.

\rev{As a matter of fact, the (weighted) Bernstein Markow property well behaves under tensorization. We omit the simple proof of the following:}
\begin{lemma}\label{lemmatensorBM}
Let $(K,w,v,\mu)$ be as above and assume the Bernstein Markow property \eqref{WBM-prop} to hold. Then, for any $\omega_1,\dots,\omega_t\in \puw{r}$, we have
\begin{equation}
\max_{x_1\in K}\dots \max_{x_t\in K}\|\omega_1(x_1)\otimes\dots\otimes\omega_t(x_t)\|_{U^t}\leq M_r(K,w,v,\mu)^t \|\otimes_{i=1}^t\omega_i\|_{v,\mu}.
\end{equation}
\end{lemma}

\subsection{Gram determinants, and \textquotedblleft free energy asymptotics\textquotedblright}
Let us introduce the notation $G_r^{v,\mu,w}$ for the weighted Gram matrix of the scalar product induced by $(v,\mu)$ \rev{on the space of weighted $U$-valued polynomials $\puw{r}$} written with respect to the canonical basis, i.e.,
\begin{equation*}
(G_r^{v,\mu,w})_{i,j}:=\int_K g(q_i(x),q_j(x)) d\mu(x).
\end{equation*}
We aim at proving the asymptotics for the determinant of $G_r^{v,\mu,w}$ \rev{as $ r \to \infty$}. As intermediate step in the proof we will compare $\det G_r^{v,\mu,w}$ with the so-called \emph{free energy} $Z_r^{v,\mu,w}$. However, in our generalized setting we need to \rev{extend} the standard definition of $Z_r^{v,\mu,w}$ (see, e.g., \cite{BlLe10}) with respect to the scalar case
$$Z_r^{v,\mu,w}:=\int_K\cdots\int_K \left\vert \left( \sum_{\sigma \in \Sigma_N} \sgn \sigma \bigotimes_{j=1}^N \left( q_{\sigma(j)} \odot w^r \right) (x_1, \ldots, x_N)\right) \right\vert^2_{g^\otimes} d\mu(x_1)\cdots d\mu(x_N),$$
where we denoted by $|\cdot|_{g^\otimes}^2$ the squared seminorm induced on $U^N$ \rev{by $g^\otimes$}. Note that the argument of the seminorm is the formal computation of the determinant of the matrix whose $(i,j)$-th element is $q_i(x_j)\odot w^r(x_j)$, where each multiplication has been replaced by a tensor product.

Now we are ready to state and prove the result of the present section.

\begin{proposition}\label{prop:MR2}
Let $(K,w,v,\mu)$ satisfy the weighted Bernstein Markow property in $\mathcal M(K,U)$, then
\begin{equation} \label{eq:freeenergy}
\lim_r \frac{n+1}{2n r N}\log \det G_r^{v,\mu,w}=\lim_r \frac{n+1}{2n r N}\log Z_r^{v,\mu,w}=\log \delta^w(K,U).
\end{equation} 
\end{proposition}
\begin{proof}
First notice that (see \cite[Eqn.s (7.3) and (7.4)]{NormSurvey}) $s \ell_r=\frac{nrN}{n+1} .$ Plugging this into \eqref{eq:defweighteddiam}, we get
\begin{equation}\label{maximizationwithexponent}
	(\delta^{w,r} (K,U))^\frac{2n r N}{n+1} = \max_{\substack{T_1 \in \mathcal M(K,U) \\ \Vert T_1 \Vert = 1}} \ldots \max_{\substack{T_N \in \mathcal M(K,U)\\ \Vert T_N \Vert = 1}} \Big| \det [T_i \left( q_j \odot w^r \right)]_{i,j=1,\dots,N} \Big|^2.
\end{equation}

By Gram Schmidt procedure we can replace in the definition of $Z_r^{v,\mu,w}$ the basis $q_j\odot w^r$ by an orthogonal basis, say $b_j\odot w^r$ with respect to the scalar product defined by $G_r^{v,\mu,w},$ \rev{see e.g. \cite[Proof of Thm. 3.1]{BlLe10},} and we obtain 
\begin{equation}\label{bygramschmidt}
\rev{Z_r^{v,\mu,w}=N!\prod_{j=1}^N\|b_j\odot w^r\|_{v,\mu}^2=N! \det G_r^{v,\mu,w}.}
\end{equation}
\rev{
Note that we can also prove that the first term equals the last one by direct computation:
\begin{align*}
&\det G_r^{v,\mu,w}= \sum_{\sigma'\in \Sigma_N}\sgn{\sigma'}\prod_{j'=1}^N\int_K g(q_{j'}\odot w^r,q_{\sigma'(j')}\odot w^r) d\mu\\
=^{{\forall \eta\in \Sigma_N}}&\sum_{\sigma'\in \Sigma_N}\sgn{\sigma'}\prod_{j=1}^N\int_K g(q_{\eta( j)}\odot w^r,q_{\sigma'\circ \eta(j)}\odot w^r) d\mu\\
=^{\forall \eta\in \Sigma_N}&\sgn{\eta}\sum_{\sigma\in \Sigma_N}\sgn{\sigma}\prod_{j=1}^N\int_K g(q_{\eta(j)}\odot w^r,q_{\sigma(j)}\odot w^r) d\mu\\
=&\frac 1 {N!}\sum_{\eta\in \Sigma_N}\sgn{\eta}\sum_{\sigma\in \Sigma_N}\sgn{\sigma}\prod_{j=1}^N\int_K g(q_{\eta(j)}\odot w^r,q_{\sigma(j)}\odot w^r) d\mu\\
=&\frac 1 {N!}\sum_{\eta\in \Sigma_N}\sgn{\eta}\sum_{\sigma\in \Sigma_N}\sgn{\sigma}\int_K\dots\int_K g^{\otimes}(\otimes_{j=1}^N q_{\eta(j)}\odot w^r,\otimes_{j=1}^N q_{\sigma(j)}\odot w^r)d\mu(x_1)\cdots d\mu(x_N)\\
=&\frac 1 {N!}\int_K\dots\int_K g^{\otimes}\left(\sum_{\eta\in \Sigma_N}\sgn{\eta}\otimes_{j=1}^N q_{\eta(j)}\odot w^r,\sum_{\sigma\in \Sigma_N}\sgn{\sigma} \otimes_{j=1}^N q_{\sigma(j)}\odot w^r\right)d\mu(x_1)\cdots d\mu(x_N)\\
=&\frac{Z_r^{v,\mu,w}}{N!}
\end{align*}
}
Let $ T_1, \ldots, \textcolor{red}{T_N} \in \mathcal M(K,U) $, with $ \Vert T_i \Vert = 1 $ for each $ i $ realizing the maximum in \eqref{maximizationwithexponent}. We then have
	\begin{align*}
		& (\delta^{w,r} (K))^\frac{2n r N}{n+1}= \left\vert \det T_i \left( q_j \odot w \right) \right\vert^2  \\
		= &   \left\vert \bigotimes_{i=1}^N T_i\left( \sum_{\sigma \in \Sigma_N} \sgn \sigma \bigotimes_{j=1}^N \left( q_{\sigma(j)} \odot w^r \right) (x_1, \ldots, x_N)\right) \right\vert^2  \\
		\leq & \left\Vert \bigotimes_{i=1}^N T_i \right\Vert_{[\mathcal M(K,U)]^N}^2 \max_{\boldsymbol{x}\in K^N} \left\Vert \sum_{\sigma \in \Sigma_N} \sgn \sigma \bigotimes_{j=1}^N \left( q_{\sigma(j)} \odot w^r \right) (x_1, \ldots, x_N) \right\Vert^2_{U^N} \\
		\leq& M_r^{2N} \left\Vert \sum_{\sigma \in \Sigma_N} \sgn \sigma \bigotimes_{j=1}^N \left( q_{\sigma(j)} \odot w^r \right) \right\Vert_{v,\mu}^2\\
		=& M_r^{2N}\int_K \cdots\int_K\left\vert\sum_{\sigma \in \Sigma_N} \sgn \sigma \bigotimes_{j=1}^N \left( q_{\sigma(j)} \odot w^r(x_1,\dots,x_N) \right) \right\vert_{g^{\otimes}}^2 d\mu(x_1)\dots d\mu(x_N)\\
		=& M_r^{2N} Z_r^{v,\mu,w}.
		\end{align*}
\rev{Note that the Bernstein Markow property and Lemma \ref{lemmatensorBM} has been used to obtain the fourth line.} Thus, extracting the $\frac{nrN}{n+1} $-th root and plugging \eqref{bygramschmidt} in, we have
\rev{\begin{equation}\label{byBM}
\delta^{w,r} (K,U)\leq M_r^{\frac{n+1}{nr}}(Z_r^{v,\mu,w})^{\frac{n+1}{2nrN}}=M_r^{\frac{n+1}{nr}}(N!)^{\frac{n+1}{2nrN}}(\det G_r^{v,\mu,w})^{\frac{n+1}{2nrN}}.
\end{equation}}
Conversely, let $(z_1,\dots,z_N)\in K^N$ such that 
$$\left\vert\sum_{\sigma \in \Sigma_N} \sgn \sigma \bigotimes_{j=1}^N \left( q_{\sigma(j)} \odot w^r(z_1,\dots,z_N) \right) \right\vert_{g^{\otimes}}^2 =\max_{\boldsymbol x\in K^N}\left\vert\sum_{\sigma \in \Sigma_N} \sgn \sigma \bigotimes_{j=1}^N \left( q_{\sigma(j)} \odot w^r(x_1,\dots,x_N) \right) \right\vert_{g^{\otimes}}^2 ,$$
and let $v_i:=v(z_i)$ and let $T_i(\omega):=( \omega(z_i),v_i)_U.$ Then we have
\begin{align*}
Z_r^{v,\mu,w}\leq& \mu(K)^N\max_{\boldsymbol x\in K^N}\left\vert\sum_{\sigma \in \Sigma_N} \sgn \sigma \bigotimes_{j=1}^N \left( q_{\sigma(j)} \odot w^r(x_1,\dots,x_N) \right) \right\vert_{g^{\otimes}}^2 \\
=&\mu(K)^N \sum_{\sigma\in \Sigma_N}\sum_{\tau\in \Sigma_N}\sign \sigma\sign\tau \prod_{i=1}^N\overline{T_i(q_{\sigma(i)}\odot w^r)}T_i(q_{\tau(i)}\odot w^r)\\
=&\mu(K)^N\det (V^H V)=\mu(K)^N|\det V|^2, 
\end{align*}
where $V_{i,j}=T_i(q_{j}\odot w^r).$ However, $T_1,\dots,T_N$ are competitors in the upper envelope defining $\delta^{w,r}(K,U)$, thus, untangling definitions, the inequality
\rev{\begin{equation}\label{byopt}
\mu(K)^{-\frac{n+1}{2nr}}(N!)^{\frac{n+1}{2nrN}}(\det G_r^{v,\mu,w})^{\frac{n+1}{2nrN}}=\mu(K)^{-\frac{n+1}{2nr}}(Z_r^{v,\mu,w})^{\frac{n+1}{2nrN}}\leq \delta^{w,r}(K,U)
\end{equation}}
follows.
We can conclude the proof by taking the limit as $r\to +\infty$ \rev{in} estimates \eqref{byBM} and \eqref{byopt}.
\end{proof}

\section{Strong Bergman Asymptotics and Convergence of Fekete vector measures}\label{sec:Fekete}

In the scalar case, i.e., $s=1$, the strong Bergman asymptotics \eqref{sba-original} is obtained in \cite{BeBoNy11} starting from the free energy asymptotics \eqref{gramasymptotic} (holding for a sequence of measures $\mu^{(r)}$ satisfying a weighted Bernstein Markov property on $K$ with respect to the weight $w=\exp(-Q)$) and performing a derivation with respect to a variation of the weight. More precisely, given a continuous test function $u$, the weight $w_t:=w\cdot\exp(-tu)$ is considered and the free energy asymptotics \eqref{gramasymptotic} is reformulated as
$$f_r(0):=-\frac {n+1}{nNr}\log \det G_r^{w_t,\mu^{(r)}}\Big|_{t=0}\to -\log\delta^{w_t}(K)\Big|_{t=0}.$$
Then, the derivatives at $t=0$ of each side is computed obtaining
\begin{align}\label{firstderivationformula}
&\frac{d}{dt}\log\delta^{w_t}(K)|_{t=0}= \frac{n+1}{n (2 \pi)^n} \sum_{l=1}^s \int_K u \ddcn{V_{K,Q}^*}\;,\\
\label{secondderivationformula}
&f'_r(0)=\frac{n+1}{n}\int_K \frac{B_{w,r}(x,\mu^{(r)})}{N}u(x) d\mu^{(r)}(x).
\end{align}
Here, $B_{w,r}(x,\mu^{(r)})=\sum_{i=1}^{N}|q_i(x;\mu^{(r)})|^2 w^{2r}$, where $\{q_i,i=1,\dots,N\}$ is an orthonormal basis of of $\mathscr P_{r,n}\C$ with respect to the product induced by $\mu$ and $w^2r.$

Finally, the limit as $r\to +\infty$ and the derivation at $t=0$ are exchanged by means of a real analysis lemma (see \cite[Lemma 6.6]{BeBo10}) exploiting the concavity of the functions $f_r$. The asymptotics for Fekete points \eqref{Feketeasympt} is obtained as a specialization of \eqref{sba-original}, by replacing the sequence of Bernstein Markov measures $\mu^{(r)}$ by the sequence of uniform probability measures supported at Fekete points.

%

We generalized the free energy asymptotics to the vector valued setting in Proposition \ref{prop:MR2}. However, the major issue in the extension of the above construction to the vector-valued framework $ s > 1 $ relies in proving the concavity of the function $ f_r $. While the derivation formula for $ f_r $ follows immediately from the scalar case, so $f'_r(0)$ and $ f''_r (0) $ can be explicitly computed (see Proposition \ref{prop:derrivatives} below), it is not clear whether the concavity of $ f_r $ follows.

Although we are not able to overcome this impasse in full generality, the problem simplifies significantly when we consider the case of asymptotically Fekete vector measures defined by vector point masses, being the vectors $N$ copies of each basis vector of $U$. This is the content of Proposition \ref{thm:MR4}, while the more general case is proposed in Conjecture \ref{thm:MR3}. 
\begin{remark}
Though a slightly weaker  version of Proposition \ref{thm:MR4} may be obtained as a consequence of the analogous scalar statement, we decided to state and directly prove Proposition \ref{thm:MR4} here. This choice has been made to offer partial support to Conjecture \ref{thm:MR3}. Indeed, it is worth noticing here that the proof of Proposition \ref{thm:MR4} goes through the same steps of \cite[Thm. C]{BeBoNy11}, while the restriction on the particular structure of the considered vector measures is exploited only when proving $f_r''(t)\leq 0$. The statement of Conjecture \ref{thm:MR3} is obtained by removing such restriction.
\end{remark}

\begin{proposition}\label{prop:derrivatives}
	Let $\omega\in \mathscr C^0(K,U^\R)$, define 
 $$w(x,t):=\left( w_1(x) \exp(-t\omega_i(x)), \ldots, w_s(x)  \exp(-t\omega_s (x))\right),$$ and let
	\begin{equation} \label{eq:fr}
		f_r(t):= - \frac{n+1}{2 nrN }\log \det G_r^{v,\mu,w(x;t)}.
	\end{equation}
	Then we have
	\begin{equation}\label{D1formulasec}
		f'_r(t)=\frac{n+1}{nN }\Re \sum_{h=1}^N \int_K g\left(b_{h}(x,t)\odot w^r(x;t),b_{h}(x,t)\odot w^r(x;t)\odot \omega(x)\right)d\mu(x)\;,
	\end{equation}
	and
	\begin{multline} 
		f''_r(t)=\frac{(n+1)r}{nN }\Bigg[\Re \sum_h\Bigg(-\int_K  g\left(b_h(x,t)\odot w^r(x;t),b_h(x,t)\odot w^r(x;t)\odot \omega^2\right) d\mu(x)\\
		+ \sum_k \Bigg(\int_K  g\left(b_h(x,t)\odot w^r(x;t)\odot \omega,b_k(x,t)\odot w^r(x;t)\right) d\mu(x) \\
		\cdot \int_K  g\left(b_h(x,t)\odot w^r(x;t),b_k(x,t)\odot w^r(x;t)\odot \omega\right) d\mu(x) \Bigg)\Bigg)\\
		+ \Bigg(\sum_h\sum_k \Big| \int_K  g\left(b_h(x,t)\odot w^r(x;t)\odot \omega,b_k(x,t)\odot w^r(x;t)\right) d\mu(x)\Big|^2\\-\sum_h\|b_h(\cdot,t)\odot\omega\|_{w^r(\cdot;t),v,\mu}^2\Bigg)\Bigg]\,,\label{D2formlasec}
	\end{multline}
	where $\{b_{1}(x,t),\dots,b_{N}(x,t)\}$ is an orthonormal basis of $\pu{r}(K)$ with respect to the scalar product defined by $(K,w(\cdot,t),v,\mu).$ 
\end{proposition} 

The proof of Proposition \ref{prop:derrivatives} is a long computation. For ease of the reader, we sketch it here and report the complete proof in Appendix \ref{sect:appendix}.

\begin{proof}[Sketch of the proof]
	First of all, by applying Jacobi Identity one computes the derivatives of $ f_r (t)$:
	$$f_r'(t)= - \frac{n+1}{2 nrN } \trace\left( (G_r^{v,\mu,w(x;t)})^{-1}\frac d{dt} G_r^{v,\mu,w(x;t)}\right),$$
	so that
	\begin{multline*}
		f''_r(t)= \frac{n+1}{2 nrN }\trace\left( (G_r^{v,\mu,w(x;t)})^{-1}\frac d{dt} G_r^{v,\mu,w(x;t)}(G_r^{v,\mu,w(x;t)})^{-1}\frac d{dt} G_r^{v,\mu,w(x;t)}\right)\\- \frac{n+1}{2 nrN }\trace((G_r^{v,\mu,w(x;t)})^{-1}\frac {d^2}{dt^2} G_r^{v,\mu,w(x;t)})\;.
	\end{multline*}
	To simplify the expression of $ f'_r (t) $ one then applies the orthogonal decomposition $$G_r^{v,\mu,w(x;t)}=P^H(t)\Lambda(t) P(t)=(\Lambda^{\frac{1}{2}}(t)P(t))^H(\Lambda^{\frac{1}{2}}(t)P(t)), $$
	which exists since $G_r^{v,\mu,w(x;t)}$ is Hermitian. Then, using the cyclic property $\trace(ABC)=\trace(BCA)$ of the trace operator, we have
	$$f_r'(t)= - \frac{n+1}{2 nrN } \trace\left((\Lambda^{-1/2}(t)P(t)) \frac d{dt} G_r^{v,\mu,w(x;t)} (\Lambda^{-1/2}(t)P(t))^H\right) .$$
	Making explicit the argument of the trace and exploiting orthogonality one obtains \eqref{D1formulasec}.
	
	The proof of \eqref{D2formlasec} is obtained by direct computation, exploiting the cyclic property and the linearity of the trace, and $ \frac d {dt} [G_r^{v,\mu,w(x;t)}]^{-1} = - [G_r^{v,\mu,w(x;t)}]^{-1} \frac d {dt} G_r^{v,\mu,w(x;t)} [G_r^{v,\mu,w(x;t)}]^{-1} $.
\end{proof}
The above results puts us in a position to prove the following:
\begin{proposition}[]\label{thm:MR4}
Let $T_{h,r}(\omega):=(u_{i(h,r)},\omega(x_{h,r}))_U$, with $h=1,\dots,N$, $x_{h,r}\in K$ (with $x_{h,r}\neq x_{k,r}$ if $h\neq k$), $r\in \N$, be a weighted asymptotically Fekete triangular array of vector measures with respect to the weight $w$, i.e.
\begin{equation} \label{eq:invertibilityVdm}
\lim_r |\det[T_{h,r}(q_k\odot w^r)]_{h,k=1,\dots,N}|^{1/(sl_r)}=\delta^{w}(K,U),
\end{equation} 
where $ i(h,r) \in \{ 1, \dots, s\} $ for any $ h = 1, \dots, r $, and $ r \in \N$. Then we have 
\begin{equation}\label{myfeketeconv}
T^{(r)}:=\frac 1 {N} \sum_{h=1}^{N} T_{h,r} \rightharpoonup^* \frac{1}{s (2\pi)^n} \sum_{l=1}^s(u_l,\cdot)_U \ddcn{V_{K,Q_l}^*},
\end{equation} 
where $Q_l:=-\log w_l$.
\end{proposition}

\begin{proof}
Without loss of generality, for $ r $ large enough, we may assume that 
$$ \{ i(1,r), \ldots, i(N,r)\} = \{ \overbrace{1, \ldots, 1}^{\text{$m_r$-terms}}, \ldots, \overbrace{s, \ldots, s}^{\text{$m_r$-terms}} \} .$$
For, note that \eqref{eq:invertibilityVdm} implies that the matrix $ V = [T_{h,r}(q_k\odot w^r)]_{h,k=1}^N $ is full rank. Indeed, recall from \eqref{eqn:basisq} that $ q_k (x) = x^{\beta(k)} u_{s(k)}$, so that, up to a permutation of the columns, $ V $ is a block diagonal matrix. The number of columns of each block is $ m_r $, whereas the number $m(j) $ of rows of the $j$-th block coincides by construction with the number of repetitions of the index $ h \in \{ 1, \ldots, N \} $ such that $ i(h,r) = j $. It is then evident that $ \det V \ne 0 $ implies that $ m(j) = m_r $, for any $ j = 1, \ldots, s $.

Let us introduce
$$ \mu^{(r)} := \frac{1}{N} \sum_{h=1}^N \delta_{x_{h,r}}, \qquad v^{(r)}(x) := \begin{cases}
    u_{i(h,r)} \ &x = x_{h,r} \\
    0 \ &x \not\in \{x_{1,r}, \ldots, x_{N,r} \}
\end{cases} ,$$
and define $ g $ as at the beginning of Section \ref{sec:BM-Gram}.  Also, let $ \omega $ and $ w(x,t) $ be as in Proposition \ref{prop:derrivatives} and define $f_r$ analogously. Notice that, due to Remark \ref{FeketeareBM}, $(K,v^{(r)},\mu^{(r)},w)$ enjoys the weighted Bernstein Markov property \eqref{BM-prop} in $\mathcal M(K,U)$. Hence, by Proposition \ref{prop:MR2}, we have
\begin{equation}
\tag{h1}\label{H1}
\lim_r f_r(0) = \log \delta^w(K,U)\,.
\end{equation}

 If we denote by $ \ell_{h,r} (x) $ the fundamental Lagrange interpolating polynomial with respect to the point $ x_{h,r} $, then
\begin{equation}
    b_{h,r} (x,t) := \frac{\sqrt{N}}{w^r (x_{h,r},t)} \ell_{h,r} (x) u_{i(h,r)},\;h=1,2,\dots,N 
\end{equation}
is an orthonormal basis with respect to the scalar product induced by $ v^{(r)}, \mu^{(r)}, w^r $. Indeed, we compute
\begin{align*}
    (b_{h,r} \odot w^r, b_{\ell, r} \odot w^r)_{v^{(r)}, \mu^{(r)}} & = \frac{1}{N} \sum_{j=1}^N \frac{N w^{2r} (x_{j,r},t)}{w^r (x_{h,r},t) w^r (x_{\ell,r},t)} \overline{(u_{i(j,r)}, u_{i(h,r)})}_U (u_{i(j,r)}, u_{i(\ell,r)})_U \\
    & = \delta_{i(j,r), i(h,r)} \delta_{i(j,r), i(\ell,r)} \delta_{h,j} \delta_{\ell,j} = \delta_{h, \ell} . 
\end{align*}
Plugging this into \eqref{D1formulasec}, we obtain
\begin{align}
    f'_r (t) & = \frac{n+1}{nN} \Re \sum_{h=1}^N \frac{1}{N} \sum_{j=1}^N \frac{N w^{2r} (x_{j,r},t)}{w^{2r} (x_{h,r},t)} \ell_{h,r}^2 (x_{j,r}) \overline{(u_{i(j,r)}, u_{i(h,r)})}_U (u_{i(j,r)}, u_{i(h,r)} \odot \omega )_U \notag\\
    & = \frac{n+1}{nN} \Re \sum_{h=1}^N \sum_{j=1}^N \frac{w^{2r} (x_{j,r},t)}{w^{2r} (x_{h,r},t)} \omega_{i(h,r)} (x_{j,r}) \delta_{h,j} \delta_{i(j,r), i(h,r)} = \frac{n+1}{nN} \Re \sum_{h=1}^N \omega_{i(h,r)} (x_{h,r}) \notag\\
    & = \frac{n+1}{n} \Re  \frac{1}{N} \sum_{h=1}^N T_{h,r} (\omega) = \frac{n+1}{n}   T^{(r)} (\omega) .\label{f1t}
\end{align}
In order to specialize \eqref{D2formlasec} to our setting, we compute separately the terms appearing in it. The first term can be computed following the lines above, and gives
\begin{equation} \label{firstterm}
-\frac{r(n+1)}{nN } \Re \sum_{h=1}^N \int_K  g\left(b_h(x,t)\odot w^r(x;t),b_h(x,t)\odot w^r(x;t)\odot\omega^2\right) d\mu(x) =  -\frac{r(n+1)}{nN} \sum_{h=1}^N \omega_{i(h,r)}^2 (x_{h,r}) .
\end{equation}
We then compute the second term of eq. \eqref{D2formlasec}:
\small
\begin{align}
    & \begin{aligned}[t]
    \frac{r(n+1)}{nN}\Re \sum_{h=1}^N \sum_{k=1}^N \Bigg(  \int_K  g\left(b_h(x,t)\odot w^r(x;t)\odot \omega,b_k(x,t)\odot w^r(x;t)\right) d\mu(x) \\
		\cdot \int_K  g\left(b_h(x,t)\odot w^r(x;t),b_k(x,t)\odot w^r(x;t)\odot \omega\right) d\mu(x) \Bigg)
  \end{aligned}\notag
  \\ & \begin{aligned}[t]
      = \frac{r(n+1)}{nN}\Re \sum_{h=1}^N \sum_{k=1}^N \Bigg(  \sum_{j_1 = 1}^N \frac{N w^{2r} (x_{j_1,r})}{N w^r (x_{h,r}) w^r (x_{k,r})} \overline{(u_{i(j_1,r)}, u_{i(h,r)}  \odot \omega )}_U (u_{i(j_1,r)}, u_{i(k,r)})_U \ell_{h,r} (x_{j_1, r}) \ell_{k,r} (x_{j_1,r}) \\
      \cdot \sum_{j_2 = 1}^N \frac{N w^{2r} (x_{j_2,r})}{N w^r (x_{h,r}) w^r (x_{k,r})} \overline{(u_{i(j_2,r)}, u_{i(h,r)})}_U (u_{i(j_2,r)}, u_{i(k,r)} \odot \omega )_U \ell_{h,r} (x_{j_2, r}) \ell_{k,r} (x_{j_2,r})
      \Bigg)
  \end{aligned}\notag
    \\ & \begin{aligned}[t]
      = \frac{r(n+1)}{nN}\Re \sum_{h=1}^N \sum_{k=1}^N \sum_{j_1 = 1}^N  \sum_{j_2 = 1}^N \frac{w^{2r} (x_{j_1,r}) w^{2r} (x_{j_2,r})}{w^{2r} (x_{h,r}) w^{2r} (x_{k,r})}  \omega_{i(h,r)} \omega_{i(k,r)}  \delta_{i(j_1,r),i(h,r)} \delta_{i(j_1,r), i(k,r)} \delta_{h,j_1} \delta_{k,j_1} \\ \delta_{i(j_2,r),i(h,r)}\delta_{i(j_2,r), i(k,r)} \delta_{h,j_2} \delta_{k,j_2}
  \end{aligned}\notag
  \\ & = \frac{r(n+1)}{nN}\sum_{h=1}^N\omega_{i(h,r)}^2\;.\label{secondterm}
\end{align}
\normalsize
Thus the sum of the first two terms vanishes. Finally, we estimate the third term of eq. \eqref{D2formlasec} by Parseval Inequality:
\begin{equation}\label{thirdterm}
\sum_h\Bigg(\sum_k \Big| \int_K  g\left(b_h(x,t)\odot w^r(x;t)\odot \omega,b_k(x,t)\odot w^r(x;t)\right) d\mu(x)\Big|^2\\-\|b_h(\cdot,t)\odot\omega\|_{w^r(\cdot;t),v,\mu}^2\Bigg)\leq 0\,.
\end{equation}
Summing equations \eqref{firstterm}, \eqref{secondterm}, and inequality \eqref{thirdterm}, we obtain
$f''_r(t)\leq 0$, i.e., 
\begin{equation}\label{H2}
\text{the functions }f_r\text{ are concave.}\tag{h2}
\end{equation}
Let us pick, for any $t\in\R$ and $r\in \N$ a vector measure $\nu^{(r,t)}\in \mathcal M(K,U)$ with representing vector field $\eta^{(r,t)}:=d\nu^{(r,t)}/d\|\nu^{(r,t)}\|$ that maximizes the functional $\mu\mapsto \det G^{\xi,\mu,w(\cdot,t)}_r$ (where $\xi=d\mu/d\|\mu\|$) among $\{\mu\in \mathcal M(K,U):\,\sup_{\omega\in \mathscr C^0(K,U)}|\mu(\omega)|/\|\omega\|_{K,U}\leq 1\}.$ Note that in the scalar case such measures are known in the literature as optimal measures \cite{BlBoLeWa10}, and have a strong connection with the theory of optimal experimental designs. Let us introduce the sequence of functions
$$\tilde f_r:=-\frac{n+1}{nrN}\log \det G^{\eta^{(r,t)},\nu^{(r,t)},w(\cdot,t)}_r.$$
Note that, for any $t$, one one hand $\tilde f_r(t)\leq f_r(t)$. On the other hand, repeating the second part of the proof of Prop. \ref{prop:MR2}, one gets $\liminf_r \tilde f_r(t) \geq \phi(t):=-\log \delta^{w(\cdot,t)}(K,U)$. Therefore we have
\begin{equation}
\label{H3}\tag{h3}
\liminf_r f_r(t)\geq \phi(t).
\end{equation}
Notice that the properties \eqref{H1}, \eqref{H2}, and \eqref{H3} are precisely the hypothesis of \cite[Lemma 6.6]{BeBo10}, whose conclusion is $\lim_r f_r'(0)=\phi'(0).$ From the scalar case it follows immediately that 
\begin{align*}
\phi'(0)=&\sum_{l=1}^s\frac d {dt}\left[\frac{1}{n(2 \pi)^n}\mathcal E(V_{K,Q_l+t\omega_l}^*,\max_j \log^+\|z_j\|)\right]_{t=0}=\frac{n+1}{sn(2\pi)^n}\sum_{l=1}^s\int_K\omega_i\ddcn{V_{K,Q}^*}\\
=& \frac{n+1}{sn(2\pi)^n}\sum_{l=1}^s\int_K(u_l,\omega)_U\ddcn{V_{K,Q}^*}.
\end{align*}
Finally, invoking \eqref{f1t}, we obtain
$$\frac{n}{n+1}\lim_r f_r'(0)=\frac{n}{n+1}\phi'(0)=\frac 1 {s(2\pi)^n}\sum_{l=1}^s\int_K(u_l,\omega)_U\ddcn{V_{K,Q}^*},$$
i.e., \eqref{myfeketeconv} holds true.
\end{proof}
We end this section by formulating the general statement of the strong Bergman asymptotics as a conjecture. 
\begin{conjecture}[Strong Bergman Asymptotics]\label{thm:MR3}
Let $(K,w,v,\mu)$ satisfy the weighted Bernstein Markov property \eqref{WBM-prop}. Let $b_1,\dots,b_{N}$ be any orthonormal basis of $\pu{r}$ with respect to the product induced by $(K,w,v,\mu)$. Then, introducing the sequence of of vector measures
$$T^{(r)}(\omega):=\int_K\Re\left( \omega,\frac{\sum_{h=1}^N\overline{(b_{h,w}(x),v)_U}b_{h,w}(x)}{N}\odot \bar{v}\right)_U d\mu(x),$$
we have
\begin{equation}
T^{(r)}\rightharpoonup^* \frac{1}{s(2 \pi)^n}\sum_{l=1}^s(\cdot,u_l)_U \ddcn{V_{K,Q_l}^*}\,,
\end{equation}
where $Q_l:=-\log w_l$ and $\ddcn{V_{K,Q}^*}$ is defined in Subsection \ref{sec:weightppt}.
The same holds true if $(K,v,\mu)$ is replaced by any Bernstein Markov sequence $(K,v^{(r)},\mu^{(r)})$.
\end{conjecture}

\section{A specialization to polynomial differential forms} \label{sect:application}

\rev{In practical applications, e.g. in the construction of numerical schemes, a prominent role is played by the case $ U = \Lambda^k_\C $, so that $ \pu{r} = \mathscr{P}_{r,n} \Lambda^k_\C $: differential forms with polynomial coefficients. When $ k = 0 $ one retrieves usual polynomial interpolation, and the theory here treated is consolidated \cite{Le06}. For higher dimensional counterparts few has been said, due to the complexity of identifying unisolvent sets, see e.g.  \cite{ABR20}. 

Let us restate our results in this perspective, noticing that the space $ \mathscr{P}_{r,n} \Lambda^k_\C  $ has itself a structure of tensor-product space $ \mathscr{P}_{r,n}\C \otimes \Lambda^k_\C $, and similarly for the real counterparts, so we have
$$ {\ddim}_{\mathbb K} \mathscr{P}_{r,n} \Lambda^k_{\mathbb K} = {\ddim}_{\mathbb K} \mathscr{P}_{r,n} \mathbb K \cdot {\ddim}_{\mathbb K} \Lambda^k_{\mathbb K} = m_r \cdot s =: \binom{n+r}{r} \binom{n}{k} = N ,$$
where $\mathbb K$ is either $\R$ or $\C$. Note that in the following we denote by $\sum'$ the summation over increasing multiindices.}

Proposition \ref{thm:MR4} with $w(x)\equiv(1,\dots,1)^t$ in this setting read as:
\begin{corollary} \label{cor:specializationforms}
Let $T_{h,r}(\omega):= \omega^{\alpha(h,r)}(x_{h,r})$, for any $ \omega(x) = \sum_{|\alpha| =k }'\omega^{\alpha} d x^{\alpha} $, with $h=1,\dots,N$, where $ x_{h,r}\in K$ (with $x_{h,r}\neq x_{k,r}$ if $h\neq k$), $r\in \N$, be an asymptotically Fekete triangular array of currents of order zero, i.e.
\begin{equation} \label{eq:feketeforms}
\lim_r |\det[T_{h,r}(q_k)]_{h,k=1,\dots,N}|^{1/(s\ell_r)}=\delta(K,\Lambda^k_\C),
\end{equation} 
and $\alpha(h,r)\in\{\alpha\in \N^s=\N^{\binom{n}{k}}: |\alpha|=k\}$ for any $h=1,\dots,N$, and $r\in N$.
Then, for any $ \omega \in \mathscr{C}^0 (K, \Lambda^k) $, we have
$$ \frac{1}{N} \sum_{h=1}^N T_{h,r} (\omega) \longrightarrow {\sum_{|\alpha| =k}}' \int_{K} \omega^\alpha \ddcn{V_K^*} \quad \text{ for } r \to \infty.$$
\end{corollary}

In the context of polynomial differential forms, Lagrange interpolation is customarily obtained by replacing nodal evaluations with integral on $k$-manifolds. This yields a particular class of currents called \emph{weights} \cite{Bossavit} (we apologize for the unfortunate coincidence of nomenclature, which we inherit from the literature)
\begin{equation} \label{eq:integralcurrent}
T : \quad \omega \mapsto \int_{\mathcal{S}} \omega dx ,
\end{equation}
being $ \mathcal{S} $ an appropriate domain of integration and $ dx $ the $k$-dimensional Hausdorff measure restricted to $ \mathcal{S} $. That is, when $ \mathcal{S} $ is a $k$-rectifiable manifold, \eqref{eq:integralcurrent} is the \emph{current of integration} on $ \mathcal{S} $. In the literature, this class of currents is restricted to particular examples of rectifiable manifolds, such as $ k $-simplices or, in a slightly more general context, $k$-cells. Also, in weights-based interpolation, the functions $ w (x) $ are all taken as constant and equal to $ 1 $. This justifies the absence of such a term in Eq. \eqref{eq:feketeforms}, as all the $ s $ terms at the right hand side of \eqref{eq:transfinitediameter} then coincide. Notwithstanding the simplification on the shape of the domain $ \mathcal{S} $, the analysis of the numerical features of integral currents as in \eqref{eq:integralcurrent} is generally hard and mostly conjectural \cite{BruniThesis} and theoretical results such as rate of convergence of weights-based interpolation of differential forms are proven only when the dimension of the ambient space is $ 1 $, see \cite{BE23}. Nevertheless, the Fekete problem studied in Corollary \ref{cor:specializationforms} emerges in the framework of polynomial differential forms; in fact, bases associated with Fekete currents have small norm, hence they define a stable interpolator.

In the search for integral currents maximizing the Vandermonde determinant \eqref{eq:feketeforms}, currents represented by vector point-masses (i.e. $T(\omega)=\sum_{j=1}^M(v,\omega(x_j))_{\Lambda^k}\mu_j$, with $\|v_j\|_{\Lambda^k}=1$, $\mu_j>0$, and $\sum_j \mu_j=1$) arise naturally as limits of currents of integration \eqref{eq:integralcurrent} normalized through the measure $ | \mathcal{S} | $, see \cite[Section 4.4.2]{BE24}. In particular, it has been observed taht, already in the one-dimensional setting, such currents of integration tend to shrink their supports toward points. Such nodes are precisely Fekete nodes, and the current of integration can hence be represented by a pair given by a point (which basically acts on the polynomial part as evaluation) and a $k$-vector (which, roughly speaking, acts on the basis of the cotangent space at the point). Notice that the direction of the vector attached to the point need not be unique; this is however not completely unexpected, as whenever we integrate we are choosing a coordinate representation for differential forms (see, for instance, the discussion in \cite[Section $3.1$]{BBZ22}). We end the section with the following interesting questions.

\begin{openproblem*}
Do all Fekete currents for polynomial differential forms degenerate to vector point-masses? Can we find Fekete currents for polynomial differential forms that are $k$-rectifiable?
\end{openproblem*}

\appendix

\section{Details of the proof of Proposition \ref{prop:derrivatives}} \label{sect:appendix}

In the following we give a detailed proof of Proposition \ref{prop:derrivatives}, expanding computations. In particular, we show that under the hypotheses of such a proposition, we have
	\begin{equation}\label{D1formula}
	f'_r(t)=\frac{n+1}{nrN }\Re \sum_{h=1}^N \int_K g\left(b_{h}(x,t)\odot w^r(x;t),b_{h}(x,t)\odot w^r(x;t)\odot \omega(x)\right)d\mu(x)\;,
\end{equation}
and
	\begin{multline} 
		f''_r(t)=\frac{(n+1)r}{nN }\Bigg[\Re \sum_h\Bigg(-\int_K  g\left(b_h(x,t)\odot w^r(x;t),b_h(x,t)\odot w^r(x;t)\odot \omega^2\right) d\mu(x)\\
		+ \sum_k \Bigg(\int_K  g\left(b_h(x,t)\odot w^r(x;t)\odot \omega,b_k(x,t)\odot w^r(x;t)\right) d\mu(x) \\
		\cdot \int_K  g\left(b_h(x,t)\odot w^r(x;t),b_k(x,t)\odot w^r(x;t)\odot \omega\right) d\mu(x) \Bigg)\Bigg)\\
		+ \Bigg(\sum_h\sum_k \Big| \int_K  g\left(b_h(x,t)\odot w^r(x;t)\odot \omega,b_k(x,t)\odot w^r(x;t)\right) d\mu(x)\Big|^2\\-\sum_h\|b_h(\cdot,t)\odot\omega\|_{w^r(\cdot;t),v,\mu}^2\Bigg)\Bigg]\,,\label{D2formla}
		\end{multline}
where $\{b_{1}(x,t),\dots,b_{N}(x,t)\}$ is an orthonormal basis of $\pu{r}(K)$ with respect to the scalar product defined by $(K,w(\cdot,t),v,\mu).$ 

\begin{proof}[Proof of Proposition \ref{prop:derrivatives}]
	By Jacobi Identity it follows that
	$$f_r'(t)= - \frac{n+1}{2 nrN } \trace\left( (G_r^{v,\mu,w(x;t)})^{-1}\frac d{dt} G_r^{v,\mu,w(x;t)}\right),$$
	so that
	\begin{multline*}
		f''_r(t)= \frac{n+1}{2 nrN }\trace\left( (G_r^{v,\mu,w(x;t)})^{-1}\frac d{dt} G_r^{v,\mu,w(x;t)}(G_r^{v,\mu,w(x;t)})^{-1}\frac d{dt} G_r^{v,\mu,w(x;t)}\right)\\- \frac{n+1}{2 nrN }\trace((G_r^{v,\mu,w(x;t)})^{-1}\frac {d^2}{dt^2} G_r^{v,\mu,w(x;t)})\;.
	\end{multline*}
	Let $G_r^{v,\mu,w(x;t)}=P^H(t)\Lambda(t) P(t)=(\Lambda^{\frac{1}{2}}(t)P(t))^H(\Lambda^{\frac{1}{2}}(t)P(t))$ be the orthogonal decomposition of the Hermitian matrix $G_r^{v,\mu,w(x;t)}$. It follows that $(G_r^{v,\mu,w(x;t)})^{-1}=(\Lambda^{-1/2}(t)P(t))^H(\Lambda^{-1/2}(t)P(t))$, so that, using the cyclic property $\trace(ABC)=\trace(BCA)$ of the trace operator, we have
	$$f_r'(t)= - \frac{n+1}{2 nrN } \trace\left((\Lambda^{-1/2}(t)P(t)) \frac d{dt} G_r^{v,\mu,w(x;t)} (\Lambda^{-1/2}(t)P(t))^H\right)=:- \frac{n+1}{2 nrN } \trace A^{(1)}.$$
	On the other hand, again using $\trace(ABC)=\trace(BCA)$,
	\begin{equation}\label{formofD2}
		f''_r(t)= \frac{n+1}{nrN } \left(\trace (A^{(1)}A^{(1)})- \trace A^{(2)}\right),
	\end{equation}
	where 
	$$A^{(2)}:= (\Lambda^{-1/2}(t)P(t)) \frac {d^2}{dt^2} G_r^{v,\mu,w(x;t)} (\Lambda^{-1/2}(t)P(t))^H\;.$$
	We need to compute the trace of the matrices $A^{(1)},$ $A^{(1)}A^{(1)}$, and $A^{(2)}.$
	\begin{align*}
		&\frac d{dt} [G_r^{v,\mu,w(x;t)}]_{h,k}= \frac d{dt}\int_K g(q_h(x)\odot w^r(x;t),q_k(x)\odot w^r(x;t))d\mu(x)\\
		=&  \int_K  \left(g\left(q_h(x)\odot w^r(x;t),q_k(x)\odot \frac d {dt}w^r(x;t)\right)+g\left(q_h(x)\odot w^r(x;t)\odot \frac d {dt}w^r(x;t),q_k(x)\right)\right)d\mu(x).
	\end{align*}
	Since $\frac{d}{dt} w^r(x;t)= - r w^r(x;t)\odot \omega$ we can conclude that
	\begin{multline*}
		\frac d{dt} [G_r^{v,\mu,w(t)}]_{h,k}= - r \int_K  g\left(q_h(x)\odot w^r(x;t),q_k(x)\odot w^r(x;t)\odot \omega\right) d\mu(x)\\
		-r \int_Kg\left(q_h(x)\odot w^r(x;t)\odot\omega,q_k(x)\odot w^r(x;t)\right)d\mu(x).
	\end{multline*}
	Using the fact that the matrix $(\Lambda^{-1/2}P)$ is the change of basis transforming the basis given by the $q_h$'s in the basis $b_{i}(\cdot,t)$ orthonormal with respect to the product induced by $(K,w(x;t),v,\mu)$ , we get
	\begin{multline}\label{A1}
		A^{(1)}_{h,k}= -r \int_K  g\left(b_h(x,t)\odot w^r(x;t),b_k(x,t)\odot w^r(x;t)\odot \omega\right) d\mu(x)\\
		-r \int_Kg\left(b_h(x,t)\odot w^r(x;t)\odot\omega,b_k(x,t)\odot w^r(x;t)\right)d\mu(x)\;.
	\end{multline}
	Note that in particular we have
	\begin{equation}\label{A1diag}
		A^{(1)}_{h,h}= -2r \Re\int_K  g\left(b_h(x,t)\odot w^r(x;t),b_h(x,t)\odot w^r(x;t)\odot \omega\right) d\mu(x))\;.
	\end{equation}
	Thus $\trace A^{(1)}=- 2\Re \sum_{h=1}^N \int_K g\left(b_{h}(x,t)\odot w^r(x,t),b_{h}(x,t)\odot w^r(x,t)\odot \omega(x)\right)d\mu(x)$, and \eqref{D1formula} follows.
	
	Let us compute $\trace(A^{(2)})$. Using again $\frac{d}{dt} w^r(x;t)= - r w^r(x;t)\odot \omega$ we can write
	\begin{multline*}
		\frac {d^2}{dt^2} [G_r^{v,\mu,w(t)}]_{h,k}= r^2 \Bigg(\int_K  g\left(q_h(x)\odot w^r(x;t)\odot\omega,q_k(x)\odot w^r(x;t)\odot \omega\right) d\mu(x)\\
		+\int_K  g\left(q_h(x)\odot w^r(x;t),q_k(x)\odot w^r(x;t)\odot \omega\odot\omega\right) d\mu(x)\\
		+ \int_Kg\left(q_h(x)\odot w^r(x;t)\odot\omega,q_k(x)\odot w^r(x;t)\odot\omega\right)d\mu(x)\\
		+\int_Kg\left(q_h(x)\odot w^r(x;t)\odot\omega\odot\omega,q_k(x)\odot w^r(x;t)\right)d\mu(x)\Bigg).
	\end{multline*}
	Therefore we have
	\begin{align*}
		A^{(2)}_{h,h}=& \begin{aligned}[t]r^2 \Bigg(\int_K  g\left(b_{h}(x,t)\odot w^r(x,t)\odot\omega,b_{h}(x,t)\odot w^r(x,t)\odot \omega\right) d\mu(x)\\
			+\int_K  g\left(b_{h}(x,t)\odot w^r(x,t),b_{h}(x,t)\odot w^r(x,t)\odot \omega^2\right) d\mu(x)\\
			+ \int_K g\left(b_{h}(x,t)\odot w^r(x,t)\odot\omega,b_{h}(x,t)\odot w^r(x,t)\odot\omega\right)d\mu(x)\\
			+\int_K g\left(b_{h}(x,t)\odot w^r(x,t)\odot\omega^2,b_{h}(x,t)\odot w^r(x,t)\right)d\mu(x)\Bigg)\end{aligned}\\
		=&2r^2 \begin{aligned}[t] \Bigg(\int_K  g\left(b_{h}(x,t)\odot w^r(x,t)\odot\omega,b_{h}(x,t)\odot w^r(x,t)\odot \omega\right) d\mu(x)\\
			+\Re\int_K  g\left(b_{h}(x,t)\odot w^r(x,t),b_{h}(x,t)\odot w^r(x,t)\odot \omega^2\right) d\mu(x)\Bigg).\end{aligned}\;,
	\end{align*}
	so that
	\begin{multline}\label{traceA2}
		\trace A^{(2)}=2r^2 \Bigg(\sum_h\|b_h(\cdot,t)\odot\omega\|_{w^r(\cdot;t),v,\mu}^2 \\+\sum_h\Re\int_K  g\left(b_{h}(x,t)\odot w^r(x,t),b_{h}(x,t)\odot w^r(x,t)\odot \omega^2\right) d\mu(x)\Bigg).
	\end{multline}

	We now compute the trace of $A^{(1)}\times A^{(1)}$ using \eqref{A1}.
	\begin{align}
		&\trace(A^{(1)}\times A^{(1)})=\sum_h\sum_k A_{h,k}^{(1)}A_{k,h}^{(1)}=\sum_h\sum_k A_{h,k}^{(1)}\overline{A_{h,k}^{(1)}}=\sum_h\sum_k |A_{h,k}^{(1)}|^2\notag\\
		=& \begin{aligned}[t]r^2\sum_h\sum_k \Big| \int_K  g\left(b_h^{(t)}(x)\odot w^r(x;t),b_k(x,t)\odot w^r(x;t)\odot \omega\right) d\mu(x)\\
			+ \int_Kg\left(b_h(x,t)\odot w^r(x;t)\odot\omega,b_k(x,t)\odot w^r(x;t)\right)d\mu(x)\Big|^2
		\end{aligned}\notag\\
		=&\begin{aligned}[t]r^2\sum_h\sum_k \Big| \int_K  g\left(b_h(x,t)\odot w^r(x;t)\odot \omega,b_k(x,t)\odot w^r(x;t)\right) d\mu(x)\Big|^2\\
			+r^2\sum_h\sum_k \Big| \int_K  g\left(b_h(x,t)\odot w^r(x;t),b_k(x,t)\odot w^r(x;t)\odot \omega\right) d\mu(x)\Big|^2\\
			+2r^2\Re\sum_h\sum_k \Bigg(\int_K  g\left(b_h(x,t)\odot w^r(x;t)\odot \omega,b_k(x,t)\odot w^r(x;t)\right) d\mu(x) \\
			\cdot \int_K  g\left(b_h(x,t)\odot w^r(x;t),b_k(x,t)\odot w^r(x;t)\odot \omega\right) d\mu(x) \Bigg)\end{aligned}\notag\\
		=&\begin{aligned}[t]2r^2\sum_h\sum_k \Big| \int_K  g\left(b_h(x,t)\odot w^r(x;t)\odot \omega,b_k(x,t)\odot w^r(x;t)\right) d\mu(x)\Big|^2\\
			+2r^2\Re\sum_h\sum_k \Bigg(\int_K  g\left(b_h(x,t)\odot w^r(x;t)\odot \omega,b_k(x,t)\odot w^r(x;t)\right) d\mu(x) \\
			\cdot \int_K  g\left(b_h(x,t)\odot w^r(x;t),b_k(x,t)\odot w^r(x;t)\odot \omega\right) d\mu(x) \Bigg)\end{aligned}\;.\label{traceA1A1}
	\end{align}
	Therefore, subtracting \eqref{traceA2} to \eqref{traceA1A1}, we have
	\begin{multline}
		\trace (A^{(1)}\times A^{(1)})-\trace(A^{(2)})=\\
		2r^2\Re \sum_h\Bigg(-\int_K  g\left(b_h(x,t)\odot w^r(x;t),b_h(x,t)\odot w^r(x;t)\odot \omega^2\right) d\mu(x)\\
		+ \sum_k \Bigg(\int_K  g\left(b_h(x,t)\odot w^r(x;t)\odot \omega,b_k(x,t)\odot w^r(x;t)\right) d\mu(x) \\
		\cdot \int_K  g\left(b_h(x,t)\odot w^r(x;t),b_k(x,t)\odot w^r(x;t)\odot \omega\right) d\mu(x) \Bigg)\\
		+ 2r^2\Bigg(\sum_h\sum_k \Big| \int_K  g\left(b_h(x,t)\odot w^r(x;t)\odot \omega,b_k(x,t)\odot w^r(x;t)\right) d\mu(x)\Big|^2\\
		-\sum_h\|b_h(\cdot,t)\odot\omega\|_{w^r(\cdot;t),v,\mu}^2\Bigg)\,,
	\end{multline}
	from which, recalling \eqref{formofD2}, equation \eqref{D2formla} follows. 
\end{proof}

\bibliographystyle{abbrv}
\bibliography{bibliography}

\begin{thebibliography}{10}

\bibitem{AMRBook}
R.~Abraham, J.~E. Marsden, and T.~Ratiu.
\newblock {\em Manifolds, tensor analysis, and applications}, volume~75 of {\em
  Applied Mathematical Sciences}.
\newblock Springer-Verlag, New York, second edition, 1988.

\bibitem{ABR20}
A.~Alonso~Rodr\'{\i}guez, L.~Bruni~Bruno, and F.~Rapetti.
\newblock Towards nonuniform distributions of unisolvent weights for high-order
  {W}hitney edge elements.
\newblock {\em Calcolo}, 59(4):Paper No. 37, 29, 2022.

\bibitem{AFW}
D.~N. Arnold, R.~S. Falk, and R.~Winther.
\newblock Finite element exterior calculus, homological techniques, and
  applications.
\newblock {\em Acta Numer.}, 15:1--155, 2006.

\bibitem{BaBi14}
M.~Baran and L.~Bialas-Ciez.
\newblock H\"older continuity of the {G}reen function and {M}arkov brothers'
  inequality.
\newblock {\em Constr. Approx.}, 40(1):121--140, 2014.

\bibitem{BaBi13}
M.~Baran, L.~Bia{\l}as-Cie{\.z}, and B.~Mil{\'o}wka.
\newblock On the best exponent in {M}arkov inequality.
\newblock {\em Potential Anal.}, 38(2):635--651, 2013.

\bibitem{BaBlLeLu19}
T.~Bayraktar, T.~Bloom, N.~Levenberg, and C.~H. Lu.
\newblock Pluripotential theory and convex bodies: large deviation principle.
\newblock {\em Ark. Mat.}, 57(2):247--283, 2019.

\bibitem{BeTa82}
E.~Bedford and B.~A. Taylor.
\newblock A new capacity for plurisubharmonic functions.
\newblock {\em Acta Math.}, 149(1-2):1--40, 1982.

\bibitem{BeBo10}
R.~Berman and S.~Boucksom.
\newblock Growth of balls of holomorphic sections and energy at equilibrium.
\newblock {\em Invent. Math.}, 181(2):337--394, 2010.

\bibitem{BeBoNy11}
R.~Berman, S.~Boucksom, and D.~Witt~Nystr\"{o}m.
\newblock Fekete points and convergence towards equilibrium measures on complex
  manifolds.
\newblock {\em Acta Math.}, 207(1):1--27, 2011.

\bibitem{BeR19}
R.~J. Berman.
\newblock Statistical mechanics of interpolation nodes, pluripotential theory
  and complex geometry.
\newblock {\em Ann. Polon. Math.}, 123(1):71--153, 2019.

\bibitem{Bl97}
T.~Bloom.
\newblock Orthogonal polynomials in {$\bold C^n$}.
\newblock {\em Indiana Univ. Math. J.}, 46(2):427--452, 1997.

\bibitem{BlBoLeWa10}
T.~Bloom, L.~Bos, N.~Levenberg, and S.~Waldron.
\newblock On the convergence of optimal measures.
\newblock {\em Constr. Approx.}, 32(1):159--179, 2010.

\bibitem{BlBoCaLe12}
T.~Bloom, L.~P. Bos, J.-P. Calvi, and N.~Levenberg.
\newblock Polynomial interpolation and approximation in {$\Bbb C^d$}.
\newblock {\em Ann. Polon. Math.}, 106:53--81, 2012.

\bibitem{BlLe03}
T.~Bloom and N.~Levenberg.
\newblock Weighted pluripotential theory in {$\bold C^N$}.
\newblock {\em Amer. J. Math.}, 125(1):57--103, 2003.

\bibitem{BlLe10}
T.~Bloom and N.~Levenberg.
\newblock Transfinite diameter notions in {$\Bbb C^N$} and integrals of
  {V}andermonde determinants.
\newblock {\em Ark. Mat.}, 48(1):17--40, 2010.

\bibitem{BlLePiWi15}
T.~Bloom, N.~Levenberg, F.~Piazzon, and F.~Wielonsky.
\newblock Bernstein-{M}arkov: a survey.
\newblock {\em Dolomites Res. Notes Approx.}, 8(2, Special Issue):75--91, 2015.

\bibitem{BoCaLeSoVi11}
L.~Bos, J.-P. Calvi, N.~Levenberg, A.~Sommariva, and M.~Vianello.
\newblock Geometric weakly admissible meshes, discrete least squares
  approximations and approximate {F}ekete points.
\newblock {\em Math. Comp.}, 80(275):1623--1638, 2011.

\bibitem{BoDeSoVi11}
L.~Bos, S.~De~Marchi, A.~Sommariva, and M.~Vianello.
\newblock Weakly admissible meshes and discrete extremal sets.
\newblock {\em Numer. Math. Theory Methods Appl.}, 4(1):1--12, 2011.

\bibitem{BoPiVi20}
L.~Bos, F.~Piazzon, and M.~Vianello.
\newblock Near {G}-optimal {T}chakaloff designs.
\newblock {\em Comput. Statist.}, 35(2):803--819, 2020.

\bibitem{Bossavit}
A.~Bossavit.
\newblock {\em Computational electromagnetism}.
\newblock Electromagnetism. Academic Press, Inc., San Diego, CA, 1998.
\newblock Variational formulations, complementarity, edge elements.

\bibitem{BruniThesis}
L.~Bruni~Bruno.
\newblock {\em Weights as degrees of freedoom for high order Whitney finite
  elements}.
\newblock PhD thesis, University of Trento, 2022.
\newblock Available at: \url{https://theses.hal.science/tel-04067201/}.

\bibitem{BE24}
L.~Bruni~Bruno and W.~Erb.
\newblock The {F}ekete problem in segmental polynomial interpolation, 2024.
\newblock preprint, available at \url{https://arxiv.org/pdf/2403.09378}.

\bibitem{BE23}
L.~Bruni~Bruno and W.~Erb.
\newblock Polynomial interpolation of function averages on interval segments.
\newblock {\em SIAM J. Numer. Anal.}, 62(4), 2024.

\bibitem{BBZ22}
L.~Bruni~Bruno and E.~Zampa.
\newblock Unisolvent and minimal physical degrees of freedom for the second
  family of polynomial differential forms.
\newblock {\em ESAIM Math. Model. Numer. Anal.}, 56(6):2239--2253, 2022.

\bibitem{Federer}
H.~Federer.
\newblock {\em Geometric measure theory}, volume Band 153 of {\em Die
  Grundlehren der mathematischen Wissenschaften}.
\newblock Springer-Verlag New York, Inc., New York, 1969.

\bibitem{Flanders}
H.~Flanders.
\newblock {\em Differential forms with applications to the physical sciences}.
\newblock Dover Books on Advanced Mathematics. Dover Publications, Inc., New
  York, second edition, 1989.

\bibitem{Hiemstra}
R.~R. Hiemstra, D.~Toshniwal, R.~H.~M. Huijsmans, and M.~I. Gerritsma.
\newblock High order geometric methods with exact conservation properties.
\newblock {\em J. Comput. Phys.}, 257(part B):1444--1471, 2014.

\bibitem{KlM91}
M.~Klimek.
\newblock {\em Pluripotential theory}, volume~6 of {\em London Mathematical
  Society Monographs. New Series}.
\newblock The Clarendon Press, Oxford University Press, New York, 1991.
\newblock Oxford Science Publications.

\bibitem{Le06}
N.~Levenberg.
\newblock Approximation in {$\Bbb C^N$}.
\newblock {\em Surv. Approx. Theory}, 2:92--140, 2006.

\bibitem{NormSurvey}
N.~Levenberg.
\newblock Weighted pluripotential theory results of {B}erman-{B}oucksom, 2010.

\bibitem{Nedelec86}
J.-C. N\'{e}d\'{e}lec.
\newblock A new family of mixed finite elements in {${\bf R}^3$}.
\newblock {\em Numer. Math.}, 50(1):57--81, 1986.

\bibitem{Pi17}
F.~Piazzon.
\newblock Some results on the rational {B}ernstein-{M}arkov property in the
  complex plane.
\newblock {\em Comput. Methods Funct. Theory}, 17(3):405--443, 2017.

\bibitem{Pi19}
F.~Piazzon.
\newblock Laplace {B}eltrami operator in the {B}aran metric and pluripotential
  equilibrium measure: the ball, the simplex, and the sphere.
\newblock {\em Comput. Methods Funct. Theory}, 19(4):547--582, 2019.

\bibitem{RapettiBossavit}
F.~Rapetti and A.~Bossavit.
\newblock Whitney forms of higher degree.
\newblock {\em SIAM J. Numer. Anal.}, 47(3):2369--2386, 2009.

\bibitem{Warner}
F.~W. Warner.
\newblock {\em Foundations of differentiable manifolds and {L}ie groups},
  volume~94 of {\em Graduate Texts in Mathematics}.
\newblock Springer-Verlag, New York-Berlin, 1983.
\newblock Corrected reprint of the 1971 edition.

\bibitem{Za75}
V.~P. Zaharjuta.
\newblock Transfinite diameter, \v{C}eby\v{s}ev constants, and capacity for
  compacta in $\mathbb{C}^n$.
\newblock {\em Mathematics of the USSR-Sbornik}, 25(3):350--364, 1975.

\bibitem{Za12}
V.~Zakharyuta.
\newblock Transfinite diameter, chebyshev constants, and capacities in
  $\mathbb{C}^n$.
\newblock {\em Annales Polonici Mathematici}, 106(1):293--313, 2012.

\bibitem{ZeZe10}
O.~Zeitouni and S.~Zelditch.
\newblock Large deviations of empirical measures of zeros of random
  polynomials.
\newblock {\em Int. Math. Res. Not. IMRN}, (20):3935--3992, 2010.

\end{thebibliography}

\end{document}